\documentclass[10pt]{article}
\usepackage{amsmath,amsthm,amsfonts,amssymb,latexsym,graphicx,stmaryrd}

\allowdisplaybreaks[4]

\newcommand{\Extra}[1]{}

\newtheorem{theorem}{Theorem}[section]
\newtheorem{lemma}[theorem]{Lemma}
\newtheorem{proposition}[theorem]{Proposition}
\newtheorem{corollary}[theorem]{Corollary}

\theoremstyle{definition}
\newtheorem{remark}[theorem]{Remark}
\newtheorem{example}[theorem]{Example}

\newcommand{\OCMIV}{Fedorova/etal:2012ICML}

\newcommand{\OCMXV}{Vovk:arXiv1612}

\renewcommand{\OCMXV}{Vovk:1986-full}

\newcommand{\st}{:}

\newcommand{\R}{\mathbb{R}}
\newcommand{\N}{\mathbb{N}}

\DeclareMathOperator{\Prob}{\mathbb{P}}
\DeclareMathOperator{\Expect}{\mathbb{E}}
\DeclareMathOperator{\UiidP}{\mathbb{P}^{\mathrm{iid}}}
\DeclareMathOperator{\UEP}{\mathbb{P}^{\mathrm{exch}}}

\DeclareMathOperator{\UCP}{\mathbb{P}^{\mathrm{conf}}}

\DeclareMathOperator{\diid}{\mathit{d}^{\mathrm{iid}}}
\DeclareMathOperator{\dE}{\mathit{d}^{\mathrm{exch}}}
\DeclareMathOperator{\dbin}{\mathit{d}^{\mathrm{bin}}}
\DeclareMathOperator{\Diid}{\mathit{D}^{\mathrm{iid}}}

\DeclareMathOperator{\Dbin}{\mathit{D}^{\mathrm{bin}}}

\newcommand*{\KP}{\mathit{KP}}

\newcommand*{\dd}{\,\mathrm{d}}
\newcommand*{\bin}{\mathrm{bin}}

\newcommand{\FFF}{\mathcal{F}}
\newcommand{\GGG}{\mathcal{G}}

  \title{Testing randomness}
  \author{Vladimir Vovk}
  
\begin{document}
  \maketitle

\begin{abstract}
  The hypothesis of randomness is fundamental in statistical machine learning
  and in many areas of nonparametric statistics;
  it says that the observations are assumed to be independent and coming from the same unknown probability distribution.
  This hypothesis is close, in certain respects, to the hypothesis of exchangeability,
  which postulates that the distribution of the observations is invariant with respect to their permutations.
  This paper reviews known methods of testing the two hypotheses concentrating on the online mode of testing,
  when the observations arrive sequentially.
  All known online methods for testing these hypotheses are based on conformal martingales,
  which are defined and studied in detail.
  The paper emphasizes conceptual and practical aspects
  and states two kinds of results.
  Validity results limit the probability of a false alarm or the frequency of false alarms
  for various procedures based on conformal martingales,
  including conformal versions of the CUSUM and Shiryaev--Roberts procedures.
  Efficiency results establish connections between randomness, exchangeability,
  and conformal martingales.

   The version of this paper at \url{http://alrw.net} (Working Paper 24)
   is updated most often.
\end{abstract}

\section{Introduction}
\label{sec:introduction}

A standard assumption in several areas of data science
has been the assumption that the data are generated in the IID fashion,
i.e., independently from the same distribution.
This assumption is also known as the assumption of randomness
(see, e.g., \cite{Wald/Wolfowitz:1943}, \cite[Section~7.1]{Lehmann:1975} and \cite{Vovk/etal:2005book}).
In this paper we are interested in testing this assumption.

The notion of randomness has been at the centre of discussions of the foundations of probability
for at least 100 years, since Richard von Mises's 1919 article \cite{Mises:1919}.
For von Mises, random sequences (\emph{collectives} in his terminology)
served as the basis for probability theory and statistics,
and other notions, such as probability, were defined in terms of collectives.
Random sequences have been eclipsed in the foundations of mathematical probability theory by measure theory
since Kolmogorov's 1933 \emph{Grundbegriffe} \cite{Kolmogorov:1933},
but the conceptual side of randomness has been explored
in the algorithmic theory of randomness (also initiated by Kolmogorov).
We will discuss the conceptual side in Section~\ref{sec:IID-exchangeability}
and then in Appendix~\ref{sec:ATR},
but we start the main part of the paper with practical methods of detecting nonrandomness.

The most familiar mode of testing randomness in statistics is where we are given a batch of data
coming from a putative power probability measure.
In Section~\ref{sec:batch} we will see how the standard statistical tests
for real-valued observations can be adapted to more general observation spaces.

We then move on to testing randomness online,
assuming that observations arrive sequentially.
Known methods of online testing of randomness are based on so-called conformal martingales.
Conformal martingales are constructed on top of conventional machine-learning algorithms
and have been used as a means of detecting deviations from randomness
both in theoretical work
(see, e.g., \cite[Section 7.1]{Vovk/etal:2005book}, \cite{Ho/Wechsler:2010}, \cite{\OCMIV})
and in practice
(in the framework of the Microsoft Azure module on time series anomaly detection \cite{Azure}).
They provide an online measure of the amount of evidence found against the hypothesis of randomness:
if the assumption of randomness is satisfied,
a fixed nonnegative conformal martingale with a positive initial value
is not expected to increase its initial value manyfold;
on the other hand, if the hypothesis of randomness is seriously violated,
a properly designed nonnegative conformal martingale with a positive initial value
can be expected to increase its value substantially.
Correspondingly, we have two desiderata for a nonnegative conformal martingale $S$:
\begin{itemize}
\item
  \textbf{Validity} is satisfied automatically:
  $S$ is not expected to ever increase its initial value by much,
  under the hypothesis of randomness.
\item
  But we also want to have \textbf{efficiency},
  e.g., to have $S_n/S_0$ large with a high probability,
  if the hypothesis of randomness is violated.
\end{itemize}
In the language of statistical hypothesis testing,
validity corresponds to controlling the error of the first kind,
and in the Neyman--Pearson paradigm, efficiency corresponds to controlling the error of the second kind
(see, e.g., \cite{Lehmann/Romano:2005}).
In this paper a more Fisherian understanding of efficiency
(introduced in Sections~\ref{sec:IID-exchangeability}--\ref{sec:probability})
will be more useful.

\begin{figure}[tb]
  \centering
  \unitlength 0.50mm
  \begin{picture}(130,60)
    \put(5,20){\framebox(10,10)[cc]{2}}
    \put(45,20){\framebox(10,10)[cc]{3}}
    \put(75,35){\framebox(10,10)[cc]{4}}
    \put(75,5){\framebox(10,10)[cc]{5}}
    \put(115,5){\framebox(10,10)[cc]{6}}
    \put(15,25){\vector(1,0){30}}
    \put(85,10){\vector(1,0){30}}
    \put(55,30){\vector(2,1){20}}
    \put(55,20){\vector(2,-1){20}}
    \put(90,35){\makebox(50,10)[cc]{weak validity}}
    \put(130,5){\makebox(40,10)[cc]{efficiency}}
  \end{picture}
  \caption{The structure of this paper (apart from the introduction and conclusion);
    the numbers refer to Sections~\ref{sec:batch}--\ref{sec:probability}.}
  \label{fig:structure}
\end{figure}
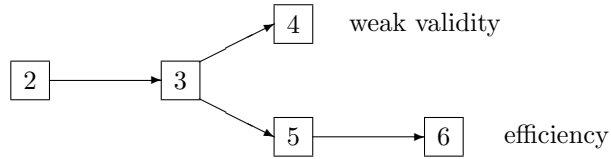

Conformal martingales are defined and their validity is established in Section~\ref{sec:martingales}.
Efficiency of a specific conformal martingale is not guaranteed
and depends on the quality of the underlying machine-learning algorithm.
Our exposition then branches in two directions,
with Section~\ref{sec:SR} relaxing the requirement of validity
and Sections~\ref{sec:IID-exchangeability}--\ref{sec:probability} discussing efficiency;
see Fig.~\ref{fig:structure}.

It is often argued that the kind of validity enjoyed by nonnegative martingales is too strong,
and we should instead be looking for a testing procedure
that is valid only in the sense of not raising false alarms too often.
In the context of testing randomness,
the interpretation of the data-generating process adopted in Section~\ref{sec:SR}
is that at first the data is IID,
but starting from some moment $T$ it ceases to be IID;
the special case $T=0$ describes the situation where the IID assumption is never satisfied
(and so this interpretation does not restrict generality).
We want our procedures to be efficient in the sense of raising an alarm soon after
the null hypothesis (such as the assumption of randomness) becomes violated;
both validity and efficiency can be required to hold with high probability or on average.
In Section~\ref{sec:SR}
conformal martingales are adapted to such less demanding requirements of validity
using the standard CUSUM and Shiryaev--Roberts procedures.

Sections~\ref{sec:IID-exchangeability} and \ref{sec:probability} discuss
the much more difficult question of efficiency
in the context of the strongly valid procedures of Section~\ref{sec:martingales}.
We ask how much we can potentially lose when using conformal martingales
as compared with unrestricted testing of either IID or exchangeability.
We will see that at a crude scale customary in the algorithmic theory of randomness
we do not lose much when restricting our attention to testing randomness with conformal martingales.

Our main running example will be the well-known USPS dataset of handwritten digits
(see, e.g., \cite[Section B.1]{Vovk/etal:2005book}),
which is known to be non-random.
In Section~\ref{sec:batch} we check that it is really non-random:
combining Bartels's ratio test with the Nearest Neighbor method
we obtain a tiny p-value.
This fact, however, is not particularly useful.
In Section~\ref{sec:martingales} we report the performance
of a nonnegative conformal martingale on the USPS dataset;
it attains a huge final value.
Such martingales are potentially very useful in practice
allowing us to decide when a trained predictor needs to be retrained.
The performance guarantees for prediction algorithms
in mainstream machine learning are proven under the assumption of randomness,
and after the deployment of such a predictor we might want to monitor
whether the new data remains IID.
As soon as it ceases to be IID, it is wise to retrain the predictor.

We will also use a much less well-known dataset
called ``Absenteeism at work'' (abbreviated to \texttt{Absenteeism})
and available at the UCI Machine Learning Repository \cite{UCI:2020}.
The data was collected from July 2007 to July 2010 at a courier company in Brazil.
We can imagine the management of the company monitoring the absences of the workforce.
If the pattern of absences loses its stability (ceases to be IID),
they might want to raise an alarm and investigate what is going on.
In Section~\ref{sec:martingales} we will construct a simple conformal martingale that finds
definitive (to use Jeffreys's \cite[Appendix B]{Jeffreys:1961} expression) evidence
against the hypothesis of randomness for the dataset.

Connections with the algorithmic theory of randomness will be explained in Appendix~\ref{sec:ATR}.
The main part of this paper will not use the algorithmic notion of randomness;
however, as customary in the algorithmic theory of randomness,
in our discussions of efficiency we will concentrate on the binary case
and on the case of a finite time horizon $N$.
These restrictions go back to Kolmogorov (cf.\ \cite[Appendix A]{\OCMXV});
it would be interesting to eliminate them after a complete exploration
of the binary and finite-horizon case (but we are not at that stage as yet).

\begin{remark}[terminology related to the hypothesis of randomness]
  In this paper we discuss two main hypotheses about the data,
  randomness and exchangeability.
  The terminology related to the former is less standardized
  and will be summarized in this remark.
  We will use constantly three closely related terms,
  ``randomness'', ``IID'', and ``power'' (used adjectivally).
  \emph{Power distributions} are those of the form $Q^N$,
  where $N$ is a natural number,
  or $Q^{\infty}$.
  Power distributions generate IID observations.
  Finally, the hypothesis of randomness means that the data-generating distribution
  is a power distribution (and so the observations are IID).
\end{remark}

\section{Conformal tests of randomness in the batch mode}
\label{sec:batch}

As already mentioned,
in this paper we are mainly interested in the \emph{online} mode of testing:
we assume that observations arrive sequentially,
and after each observation we evaluate the degree
to which the hypothesis of randomness has been discredited.
We will discuss this online setting starting from the next section,
but in this section discuss the \emph{batch} setting, which is more standard in statistics.

Let us fix $N\in\N:=\{1,2,\dots\}$, the size of the batch.
We would like to test the hypothesis of randomness using $N$ observations.
There are numerous standard tests of randomness in statistics:
see, e.g., \cite[Chapter 7]{Lehmann:1975}.
These tests, however, are usually applicable only to batches of real numbers,
whereas in this paper we are interested in more general observations.
In particular we will apply them to the USPS dataset of handwritten digits
discussed at the end of Section~\ref{sec:introduction}.
To reduce the general case to real-valued observations
we can apply basic ideas of conformal prediction \cite{Vovk/etal:2005book}.
Let $\mathbf{Z}$ be a measurable space, the space of observations.

A \emph{nonconformity measure} is a measurable function $A$
mapping any finite sequence $(z_1,\dots,z_n)\in\mathbf{Z}^n$ of observations of any length $n\in\N$
to a sequence of numbers $(\alpha_1,\dots,\alpha_n)\in\R^n$ of the same length
that is \emph{equivariant} in the following sense:
for any $n\in\N$ and any permutation $\pi:\{1,\dots,n\}\to\{1,\dots,n\}$,
\begin{equation}\label{eq:equivariance}
  A(z_1,\dots,z_n)
  =
  (\alpha_1,\dots,\alpha_n)
  \Longrightarrow
  A(z_{\pi(1)},\dots,z_{\pi(n)})
  =
  (\alpha_{\pi(1)},\dots,\alpha_{\pi(n)}).
\end{equation}
Intuitively, $\alpha_i$ (the \emph{nonconformity score} of $z_i$)
tells us how strange $z_i$ looks as an element of the sequence $(z_1,\dots,z_n)$.
(At this time the only relevant value is $n:=N$,
but in the next section we will need all $n\in\N$.)

Any conventional machine-learning algorithm can be turned
(usually in more than one way)
into a nonconformity measure.
For example, suppose each observation $z_i$ (an element of the USPS dataset) consists of two components,
$z_i=(x_i,y_i)$,
where $x_i\in[-1,1]^{16\times16}$ is a hand-written digit
(a $16\times16$ matrix of pixels, each pixel represented by its brightness on the scale $[-1,1]$)
and $y_i\in\{0,\dots,9\}$ is its label
(the true digit represented by the image).
The 1-Nearest Neighbour algorithm can be turned into the nonconformity measure
\begin{equation}\label{eq:1-NN}
  \alpha_i
  :=
  \frac
  {\min_{j\in\{1,\dots,n\}:y_j=y_i,j\ne i}d(x_i,x_j)}
  {\min_{j\in\{1,\dots,n\}:y_j\ne y_i}d(x_i,x_j)},
\end{equation}
where $d(x_i,x_j)$ is the Euclidean distance between $x_i$ and $x_j$.
Intuitively, a hand-written digit is regarded as strange
if it is closer to a digit labeled in a different way
than to digits labeled in the same way.
See, e.g., \cite{Vovk/etal:2005book,Bala/etal:2014}
for numerous other examples of nonconformity measures.

In our terminology we will follow \cite{Vovk/Wang:2020}.
A \emph{p-variable for testing randomness} in $Z^N$, where $Z$ is a measurable space,
is a measurable function $f:Z^N\to[0,1]$ such that,
for any power probability measure $P$ on $Z^N$
(i.e., for any measure $P$ of the form $Q^N$, where $Q$ is a probability measure on $Z$)
and any $\epsilon>0$,
\begin{equation}\label{eq:p-variable}
  P
  \left(
    \{(z_1,\dots,z_N)\st f(z_1,\dots,z_N)\le\epsilon\}
  \right)
  \le
  \epsilon.
\end{equation}
The value taken by $f$ on the realized sample is then a bona fide p-value
(perhaps conservative) for testing the hypothesis of randomness.
Similarly, a \emph{p-variable for testing exchangeability} in $Z^N$ is a measurable function $f:Z^N\to[0,1]$ such that,
for any exchangeable probability measure $P$ on $Z^N$
(i.e., for any measure that is invariant with respect to permutations of observations)
and any $\epsilon>0$, we have \eqref{eq:p-variable}.
The values taken by such $f$ are p-values for testing the hypothesis of exchangeability.

It is clear that every p-variable for testing exchangeability
is a p-variable for testing randomness.
Nonparametric statistics provides us with plenty of p-variables for testing exchangeability in $\R^N$
(under the rubric ``testing of randomness'', even though they in fact test the weaker assumption of exchangeability;
see, e.g., \cite{Wald/Wolfowitz:1943,Lehmann:1975,Bartels:1982}).
The following proposition shows how such a function $f$,
in combination with a nonconformity measure $A$,
generates a p-variable for testing exchangeability in $\mathbf{Z}^N$.
Set
\[
  (f\circ A)(z_1,\dots,z_N)
  :=
  f(A(z_1,\dots,z_N)).
\]

\begin{proposition}\label{prop:batch}
  If $f$ is a p-variable for testing exchangeability in $\R^N$ and $A$ is a nonconformity measure,
  then $f\circ A$ is a p-variable for testing exchangeability in $\mathbf{Z}^N$.
\end{proposition}

\begin{proof}
  This follows immediately from the equivariance property \eqref{eq:equivariance}:
  if $P$ is an exchangeable probability measure on $\mathbf{Z}^N$,
  then its pushforward $PA^{-1}$ is an exchangeable probability measure on $\R^N$,
  and so
  \[
    P(f\circ A \le \epsilon)
    =
    (PA^{-1})(f \le \epsilon)
    \le
    \epsilon.
  \]
  To check that the pushforward $PA^{-1}$ of an exchangeable probability measure $P$ on $\mathbf{Z}^N$
  is indeed exchangeable,
  it suffices to notice that,
  for any permutation $\pi:\{1,\dots,N\}\to\{1,\dots,N\}$
  and any event $E\subseteq\mathbf{Z}^N$,
  \begin{align*}
    &PA^{-1}
    \left(\left\{
      (u_1,\dots,u_N) \st \left(u_{\pi(1)},\dots,u_{\pi(N)}\right) \in E
    \right\}\right)\\
    &=
    P
    \left(\left\{
      (z_1,\dots,z_N)
      \st
      A(z_1,\dots,z_N)
      \in
      \left\{(u_1,\dots,u_N) \st \left(u_{\pi(1)},\dots,u_{\pi(N)}\right) \in E\right\}
    \right\}\right)\\
    &=
    P
    \left(\left\{
      (z_1,\dots,z_N)
      \st
      A\left(z_{\pi(1)},\dots,z_{\pi(N)}\right)
      \in
      E
    \right\}\right)\\
    &=
    P
    \left(\left\{
      \left(z_{\pi(1)},\dots,z_{\pi(N)}\right)
      \st
      A\left(z_{\pi(1)},\dots,z_{\pi(N)}\right)
      \in
      E
    \right\}\right)
    =
    PA^{-1}(E).
  \end{align*}
  (The second equality follows from the equivariance of $A$ and the third from the exchangeability of $P$.)
\end{proof}

\begin{example}\label{ex:counterexample}
  Let us check that Proposition~\ref{prop:batch} ceases to be true if ``exchangeability'' is replaced by ``randomness''.
  Suppose $\mathbf{Z}=[0,1]$,
  and define the nonconformity score $\alpha_i$ of $z_i$ in $(z_1,\dots,z_n)$ by
  \[
    \alpha_i
    :=
    \begin{cases}
      1 & \text{if $z_i\ge m_i$}\\
      0 & \text{otherwise},
    \end{cases}
  \]
  where $m_i$ is the median of the multiset $\{z_1,\dots,z_{i-1},z_{i+1},\dots,z_n\}$.
  Suppose $N$ is an even number (to simplify the notion of a median)
  and let $P:=U^N$, where $U$ is the uniform probability measure on $[0,1]$.
  Then the pushforward $PA^{-1}$ is concentrated on the subset of $\{0,1\}^N$
  containing equal numbers of 0s and 1s.
  (Roughly, this corresponds to half of the elements being above the median.
  Intuitively, $A$ transforms a sequence that looks IID
  to a sequence that does not look IID at all for a large $N$,
  since it contains equal numbers of 0s and 1s.)
  By Stirling's formula (see, e.g., \eqref{eq:Stirling} below),
  the probability of this subset is at most
  \begin{equation}\label{eq:half}
    \binom{N}{N/2}
    2^{-N}
    \sim
    \sqrt{2/\pi}
    N^{-1/2}
    <
    N^{-1/2}
  \end{equation}
  under any product measure.
  Therefore, assuming $N$ is large,
  the function $f$ taking value $N^{-1/2}$ on this subset and $1$ elsewhere on $\{0,1\}^N$
  is a p-variable for testing randomness while $f\circ A$ is not
  (indeed, $f\circ A$ will take value $N^{-1/2}$ almost surely under $P$,
  and then it is obvious that it cannot be a p-variable).
\end{example}

\begin{table}[bt]
  \caption{Some p-values for the USPS dataset and Bartels's ratio test}
  \begin{center}
    \begin{tabular}{l|cc}
        & Euclidean distance & Tangent distance \\
      \hline
      \rule{0pt}{2.5ex}p-value & $2.7\times10^{-11}$ & $7.5\times10^{-16}$
    \end{tabular}
  \end{center}
  \label{tab:batch}
\end{table}

Table~\ref{tab:batch} gives the p-values produced by Bartels's \cite{Bartels:1982} ratio test
applied to the nonconformity scores \eqref{eq:1-NN},
where $d$ is the Euclidean distance or the more sophisticated tangent distance
\cite{Simard/etal:1993,Keysers:2000}, as indicated.
The p-values are very low, especially for the tangent distance.

\begin{remark}
  It is important to keep in mind that the standard implementations of the tangent distance
  are not always symmetric and $d(x,x')\ne d(x',x)$ is possible
  (in particular, this is the case for Keysers's \cite{Keysers:2000} implementation used in this paper).
  Whereas \eqref{eq:1-NN} itself defines a nonconformity measure regardless of the symmetry of $d$,
  efficient implementations of conformal testing of randomness based on \eqref{eq:1-NN}
  tend to rely on the symmetry of $d$ and lose their validity if $d$ is not symmetric.
  This is true for the implementation used for empirical studies
  in this paper;
  one possible solution (used here) is to replace $d(x_i,x_j)$ in \eqref{eq:1-NN}
  by the arithmetic mean of $d(x_i,x_j)$ and $d(x_j,x_i)$
  (using the geometric mean produces similar results).
\end{remark}

\section{Conformal martingales}
\label{sec:martingales}

First let me give some basic definitions of conformal prediction
(see, e.g., \cite{Vovk/etal:2005book} or \cite{Bala/etal:2014} for further details).
Let us fix a nonconformity measure $A$.
If $Z$ is a set, $Z^*$ is the set of all finite sequences of elements of $Z$;
if $Z$ is a measurable space, $Z^*$ is also regarded as a measurable space.
The \emph{p-value} $p_n$ generated by $A$
after being fed with a sequence of observations $(z_1,\dots,z_n)\in\mathbf{Z}^*$
is
\begin{equation}\label{eq:p}
  p_n
  =
  p_n(z_1,\dots,z_n,\theta_n)
  :=
  \frac
  {
    \left|
      \left\{
        i \st \alpha_i>\alpha_n
      \right\}
    \right|
    +
    \theta_n
    \left|
      \left\{
        i \st \alpha_i=\alpha_n
      \right\}
    \right|
  }
  {n},
\end{equation}
where $i$ ranges over $\{1,\dots,n\}$,
$\alpha_1,\dots,\alpha_n$ are the nonconformity scores for $z_1,\dots,z_n$,
\[
  (\alpha_1,\dots,\alpha_n)
  =
  A(z_1,\dots,z_n),
\]
and $\theta_n$ is a random number distributed uniformly on the interval $[0,1]$.
The following proposition gives the standard property of validity for conformal prediction
(for a proof, see, e.g., \cite[Proposition~2.8]{Vovk/etal:2005book}).

\begin{proposition}\label{prop:validity}
  Suppose the observations $z_1,z_2,\dots$ are IID,
  $\theta_1,\theta_2,\dots$ are IID and distributed uniformly on $[0,1]$,
  and the sequences $z_1,z_2,\dots$ and $\theta_1,\theta_2,\dots$ are independent.
  Then the p-values $p_1,p_2,\dots$ as defined in \eqref{eq:p}
  are IID and distributed uniformly on $[0,1]$.
\end{proposition}

\begin{remark}\label{rem:validity}
  On a few occasions,
  we will also need a version of Proposition~\ref{prop:validity}
  for a finite horizon $N\in\N$.
  Now we have finite input sequences $z_1,\dots,z_N$ and $\theta_1,\dots,\theta_N$
  and, correspondingly, a finite output sequence $p_1,\dots,p_N$.
  The conclusion of Proposition~\ref{prop:validity} will still hold
  even if we relax the assumption of $z_1,\dots,z_N$ being IID
  to the assumption that they are exchangeable.
  (See \cite[Theorem~8.2]{Vovk/etal:2005book}.)
\end{remark}

The formal definition of a nonnegative conformal martingale
(equivalent to the definition given in \cite[Section~7.1]{Vovk/etal:2005book})
given in this paragraph will be followed by a discussion of the intuition behind it
in the following paragraph
(in our informal discussions we will often abbreviate
``nonnegative conformal martingale'' to ``conformal martingale'').
A \emph{betting martingale} is a measurable function $F:[0,1]^*\to[0,\infty]$
such that, for each sequence $(u_1,\dots,u_{n-1})\in[0,1]^{n-1}$, $n\ge1$, we have
\begin{equation}\label{eq:betting-martingale}
  \int_0^1
  F(u_1,\dots,u_{n-1},u)
  \dd u
  =
  F(u_1,\dots,u_{n-1});
\end{equation}
notice that betting martingales are required to be nonnegative.
A \emph{nonnegative conformal martingale} is any sequence of functions
$S_n:(\mathbf{Z}\times[0,1])^{\infty}\to[0,\infty]$, $n=0,1,\dots$,
such that, for some nonconformity measure $A$ and betting martingale $F$,
for all $m\in\{0,1,\dots\}$, $(z_1,z_2,\dots)\in\mathbf{Z}^{\infty}$,
and $(\theta_1,\theta_2,\dots)\in[0,1]^{\infty}$,
\begin{equation}\label{eq:S_n}
  S_m(z_1,\theta_1,z_2,\theta_2,\dots)
  =
  F(p_1,\dots,p_m),
\end{equation}
where $p_n$, $n\in\N$, is the p-value computed by~\eqref{eq:p}
from the nonconformity measure $A$,
the observations $z_1,z_2,\dots$,
and the $n$th element $\theta_n$ of the sequence $(\theta_1,\theta_2,\dots)$.
Notice that $S_m(z_1,\theta_1,z_2,\theta_2,\dots)$ depends on $z_1,\theta_1,z_2,\theta_2,\dots$
only via $z_1,\theta_1,\dots,z_m,\theta_m$.

Intuitively, a betting martingale describes the evolution of the capital
of a player who gambles against the hypothesis that the p-values $p_1,p_2,\dots$
are distributed uniformly and independently,
as they should under the hypothesis of randomness
(see Proposition~\ref{prop:validity}).
The requirement \eqref{eq:betting-martingale} expresses the fairness of the game:
at step $n-1$, the conditional expected value of the player's future capital at step $n$
given the present situation (i.e., the first $n-1$ p-values) is equal to his current capital.
This formalization of fair betting goes back to Ville \cite{Ville:1939}
and was made very popular in probability theory by Doob \cite{Doob:1953};
for a recent review of developments in various directions, see \cite{Shafer/Vovk:2019}.
A conformal martingale is what we get when we feed a betting martingale
with p-values \eqref{eq:p} produced by conformal prediction.

One way of constructing betting martingales is to use ``betting functions'',
in the terminology of \cite{\OCMIV}.
A \emph{betting function} $f:[0,1]\to[0,\infty]$ is a function satisfying $\int_0^1 f(u)\dd u=1$.
A useful method of betting against the hypothesis that the p-values $p_1,p_2,\dots$
are independent and uniformly distributed is to choose,
before each step $n$, a betting function $f_n$ that may depend on $p_1,\dots,p_{n-1}$
(in a measurable manner).
Then
\begin{equation}\label{eq:bet}
  F(p_1,\dots,p_n):=f_1(p_1)\dots f_n(p_n),
  \quad
  n=0,1,\dots,
\end{equation}
will be a betting martingale
(a conformal martingale if $p_1,p_2,\dots$ are generated by conformal prediction, \eqref{eq:p}).

\begin{remark}\label{rem:tricky}
  Conformal martingales are \emph{exchangeability martingales},
  i.e., stochastic processes that are martingales with respect to any exchangeable distribution.
  The existence of non-trivial exchangeability martingales is, however, not obvious.
  It is easy to check that for the natural underlying filtration $(\FFF_n)_{n=0,1,\dots}$
  generated by the observations $z_1,z_2,\dots$ the only exchangeability martingales
  are almost sure constants.
  There are two reasons why non-trivial exchangeability martingales exist:
  \begin{itemize}
  \item
    Our underlying filtration is poorer than $\FFF_n$.
    A conformal martingale $S$ satisfies
    \[
      \Expect(S_n \mid S_1,\dots,S_{n-1})
      =
      S_{n-1},
      \quad
      n\in\N,
    \]
    i.e., it is a martingale in its own filtration.
    Moreover, it is a martingale in the filtration $(\GGG_n)_{n=0,1,\dots}$
    where $\GGG_n$ is generated by the first $n$ p-values $p_1,\dots,p_n$.
  \item
    Conformal martingales are randomized:
    they also depend on the random numbers $\theta_1,\theta_2,\dots$.
  \end{itemize}
  The first reason alone seems to be insufficient for getting really useful exchangeability martingales:
  e.g., in the binary case $\mathbf{Z}=\{0,1\}$,
  the observations $z_1,\dots,z_n$ are determined by the values $S_0,S_1,\dots,S_n$,
  unless $S_{i}=S_{i-1}$ for some $i\in\{1,\dots,n\}$
  (let us check this for $n=1$:
  depending on the value of $z_1$, we have either $S_1(z_1,\dots)>S_0$ or $S_1(z_1,\dots)<S_0$,
  and knowing which inequality is true determines $z_1$;
  for general $n$, use induction in $n$).
  In many practically interesting cases there is not much randomness in conformal martingales;
  it is only used for tie-breaking.
  However, even a tiny amount of randomness can be conceptually important
  (other fields where this phenomenon has been observed are differential privacy
  and defensive forecasting \cite[Section~12.7]{Shafer/Vovk:2019}).
\end{remark}

\begin{remark}\label{rem:de-Finetti}
  Notice that exchangeability martingales
  discussed in Remark~\ref{rem:tricky} can be equivalently defined
  as stochastic processes that are martingales with respect to any power distribution,
  assuming that the observation space $\mathbf{Z}$ is a Borel space
  (this is a very weak requirement; see, e.g., \cite[B.3.2]{Schervish:1995}).
  Indeed, according to de Finetti's theorem
  (see, e.g., \cite[Theorem~1.49]{Schervish:1995})
  every exchangeable distribution is then a Bayesian mixture of power distributions,
  and so being a martingale with respect to all power distributions
  and with respect to all exchangeable distributions
  are equivalent.
\end{remark}

\subsection*{Using conformal martingales for testing randomness}

We only consider nonnegative conformal martingales $S$ with $S_0\in(0,\infty)$.
Let us see how such martingales can be used for testing randomness.

A possible goal is to raise an alarm warning the user about lack of randomness as soon as possible.
Ville's inequality \cite[Chapter 7, Section 3, Theorem 1.III]{Shiryaev:2019}
says that, for any $c>1$,
\[
  \Prob(\exists n: S_n/S_0\ge c) \le 1/c
\]
under any power distribution.
This means that if we raise an alarm when $S_n/S_0$ reaches threshold $c$,
we will be wrong with probability at most $1/c$.
This is a strong (in some situation too strong) requirement of validity,
and we will sometimes refer to it as \emph{strong validity}.

We can also interpret $S_n/S_0$ directly as the amount of evidence
detected against the first $n$ observations being IID.
In principle, there is no need to raise an alarm explicitly,
and we can leave the decision whether to abandon the assumption of randomness
with the user of our methods.

As an example, for the USPS dataset of handwritten digits (9298 in total),
the online performance of a nonnegative conformal martingale
based on the nonconformity measure \eqref{eq:1-NN} (with Euclidean distance)
is shown in the left panel of Fig.~\ref{fig:USPS}
(which is Fig.~7.6 in \cite{Vovk/etal:2005book},
where full details of the conformal martingale can be found).
We already know from Section~\ref{sec:batch} that the USPS dataset is not random,
and its lack of randomness is detected by this conformal martingale in the online mode.
The advantage of the online mode is that such a conformal martingale may be used in practice for deciding
when a digit classifier needs to be retrained
(we can see that approximately after the 2400th observation it would be definitely desirable).

\begin{figure}[t]
  \begin{center}
    \includegraphics[width=0.49\linewidth]{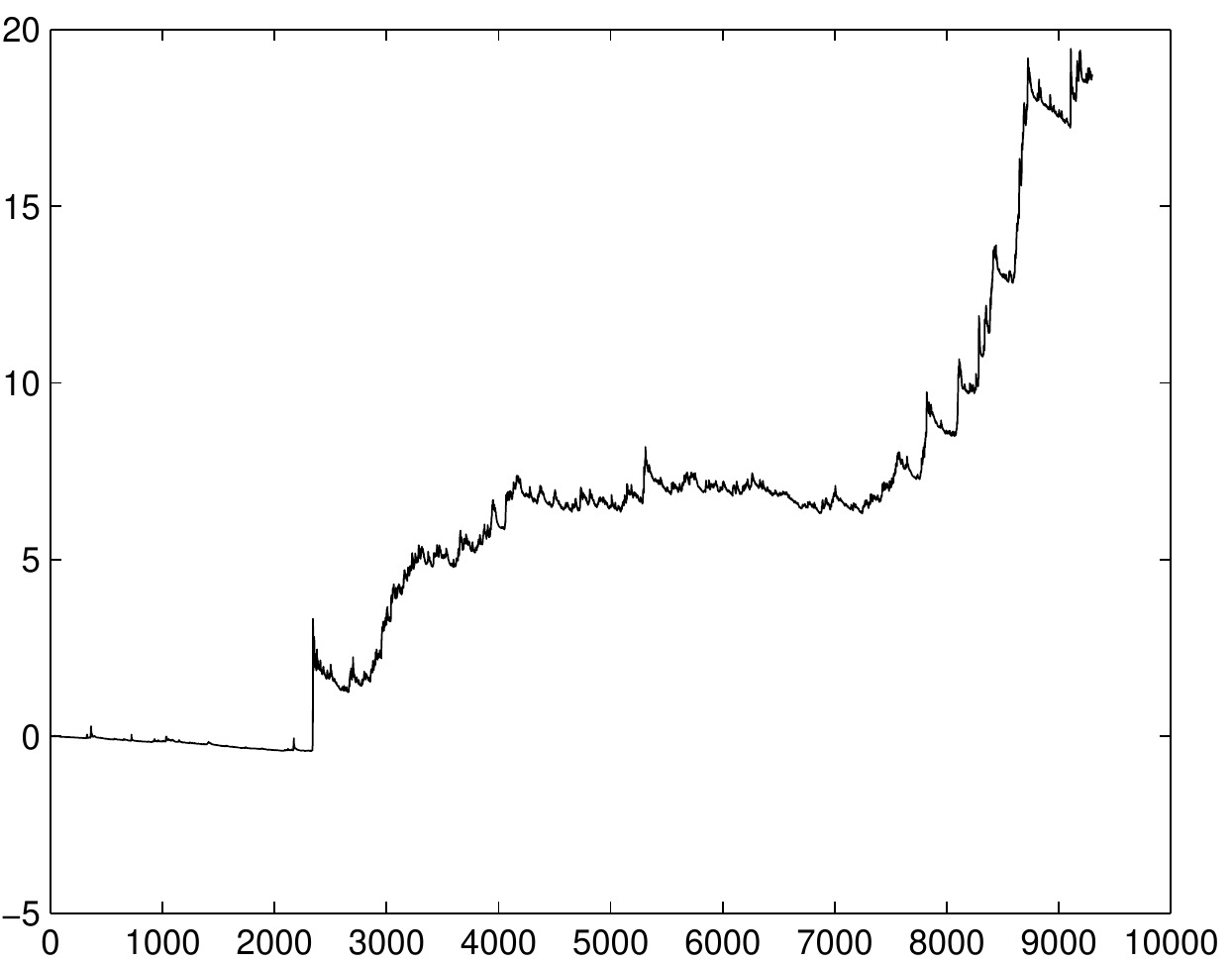}
    \includegraphics[width=0.49\linewidth]{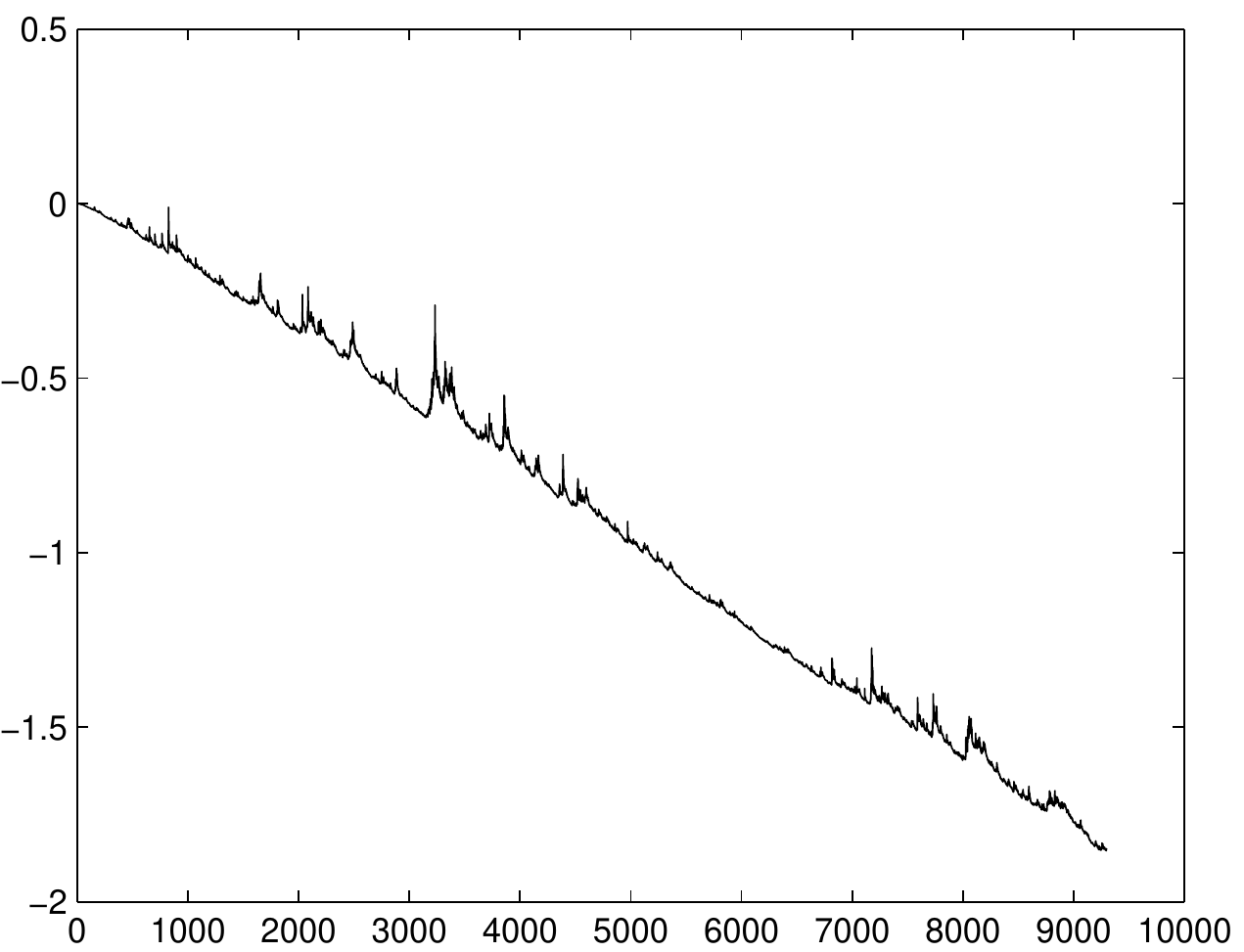}
  \end{center}
  \caption{Left panel:
    The values $S_n$ of a nonnegative conformal martingale
    after observing the first $n$ digits, $n=0,\dots,9298$, of the USPS dataset,
    with the log-10 scale for the vertical axis.
    The initial value $S_0$ is 1, and the final value $S_{9298}$ is $4.71\times10^{18}$.
    Right panel:
    The values $S_n$ of the same nonnegative conformal martingale as in the left panel
    after observing the first $n$ digits of a randomly permuted USPS dataset,
    with the log-10 scale for the vertical axis.
    The initial value $S_0$ is 1, and the final value $S_{9298}$ is 0.0142.}
  \label{fig:USPS}
\end{figure}

The conformal martingale whose performance is shown in Fig.~\ref{fig:USPS}
is fairly complicated, but the ideas behind its construction are instructive.
Any function
\begin{equation}\label{eq:betting-function}
  f^{(\kappa)}(u)
  :=
  \kappa u^{1-\kappa},
  \quad
  u\in[0,1],
  \quad
  \kappa\in(0,1),
\end{equation}
is a betting function, in the sense of satisfying $\int_0 f^{(\kappa)}(u) \dd u = 1$.
This implies that
\[
  F^{(\kappa)}(u_1,\dots,u_n)
  :=
  \prod_{i=1}^n
  f^{(\kappa)}(u_i)
\]
is a betting martingale (satisfies~\eqref{eq:betting-martingale})
and so determines, by \eqref{eq:S_n},
a conformal martingale $S=S^{(\kappa)}$.
To get rid of the dependence on $\kappa$,
we can use the conformal martingale $\int_0^1 S^{(\kappa)} \dd\kappa$,
which was called the \emph{simple mixture} in \cite[Section~7.1]{Vovk/etal:2005book}.
The simple mixture starts from 1 and its final value is $2.18\times10^{10}$.

The simple mixture attains an astronomical final value,
but it can be improved further,
in some sense tracking the best value of $\kappa$
(following Herbster and Warmuth's \cite{Herbster/Warmuth:1998ML} idea
of tracking the best expert).
For a stochastic process producing a random sequence $\kappa_1,\kappa_2,\dots$,
we can integrate
\[
  F^{(\kappa_1,\kappa_2,\dots)}(u_1,\dots,u_n)
  :=
  \prod_{i=1}^n
  f^{(\kappa_i)}(u_i)
\]
with respect to the distribution of $\kappa_1,\kappa_2,\dots$,
and for a reasonable choice of the stochastic process,
we can improve the final value of the conformal martingale.
It can be further boosted by allowing the stochastic process to ``sleep'' at some steps
(so that the corresponding conformal martingale does not gamble on those steps).
In statistical literature,
these ideas are discussed in \cite{Erven/etal:2012}.

The idea behind the betting functions \eqref{eq:betting-function}
is that the lack of randomness will show itself in abnormally low p-values.
Fedorova and Nouretdinov \cite{\OCMIV} came up with an unexpected new idea:
in fact, we can gamble against \emph{any} non-uniformity in the distribution of the p-values,
and this may be a very successful strategy for detecting non-randomness.

Suppose we know the true distribution of the $n$th p-value $p_n$
(conditional on knowing the first $n-1$ p-values),
and suppose it is continuous with density $\rho$.
What betting function $f$ should we choose?
This is a continuous version of the standard problem of horse race betting
\cite{Kelly:1956,Cover/Thomas:2006}.
The following simple lemma sheds some light on ways of exploiting such non-uniformity.

\begin{lemma}
  For any probability density functions $\rho$ and $f$ on $[0,1]$
  (so that $\int_0^1 \rho(p) \dd p = 1$ and $\int_0^1 f(p) \dd p = 1$),
  \begin{align}
    \int_0^1
    \Bigl(
      \log\rho(p)
    \Bigr)
    \rho(p) \dd p
    &\ge
    \int_0^1
    \Bigl(
      \log f(p)
    \Bigr)
    \rho(p) \dd p
    \label{eq:KL-1}\\
    \intertext{and}
    \int_0^1
    \Bigl(
      \log\rho(p)
    \Bigr)
    \rho(p) \dd p
    &\ge
    0.
    \label{eq:KL-2}
  \end{align}
\end{lemma}

\begin{proof}
  It is well known
  (and immediately follows from the inequality $\log x\le x-1$)
  that the Kullback--Leibler divergence is always non-negative:
  \[
    \int_0^1\log\Bigl(\frac{\rho(p)}{f(p)}\Bigr)\rho(p) \dd p
    \ge
    0.
  \]
  This is equivalent to \eqref{eq:KL-1}.
  And \eqref{eq:KL-2} is a special case of \eqref{eq:KL-1}
  corresponding to the probability density function $f:=1$.
\end{proof}

If we choose a betting function $f$,
the log of our capital will increase by the right-hand side of \eqref{eq:KL-1}
in expectation.
Therefore, according to \eqref{eq:KL-1},
the largest increase in expectation is achieved
when we use $\rho$ as the betting function.
(Increasing the log capital as much as possible in expectation
is a natural objective since such increases add
to give the log of the final capital,
and so we can apply the law of large numbers,
as in horse racing \cite[Section 6.1]{Cover/Thomas:2006}
or log-optimal portfolios \cite[Chapter 16]{Cover/Thomas:2006}.)
The discrete version of this strategy
is known as \emph{Kelly gambling} \cite[Theorem 6.1.2]{Cover/Thomas:2006}.

How efficient the betting function $\rho$ is depends on the left-hand side of \eqref{eq:KL-1},
which is the minus (differential) entropy of $\rho$.
The maximum entropy distribution on $[0,1]$ is the uniform distribution,
as asserted by \eqref{eq:KL-2},
whose right-hand side is equal to the minus entropy of the uniform distribution.
(This is a very special case of standard maximum entropy results,
such as \cite[Theorem 12.1.1]{Cover/Thomas:2006}.)
The uniform true distribution for the p-values
gives zero expected increase in the log capital;
otherwise, it is positive.

\begin{figure}[t]
  \begin{center}
    \includegraphics[width=0.49\linewidth]{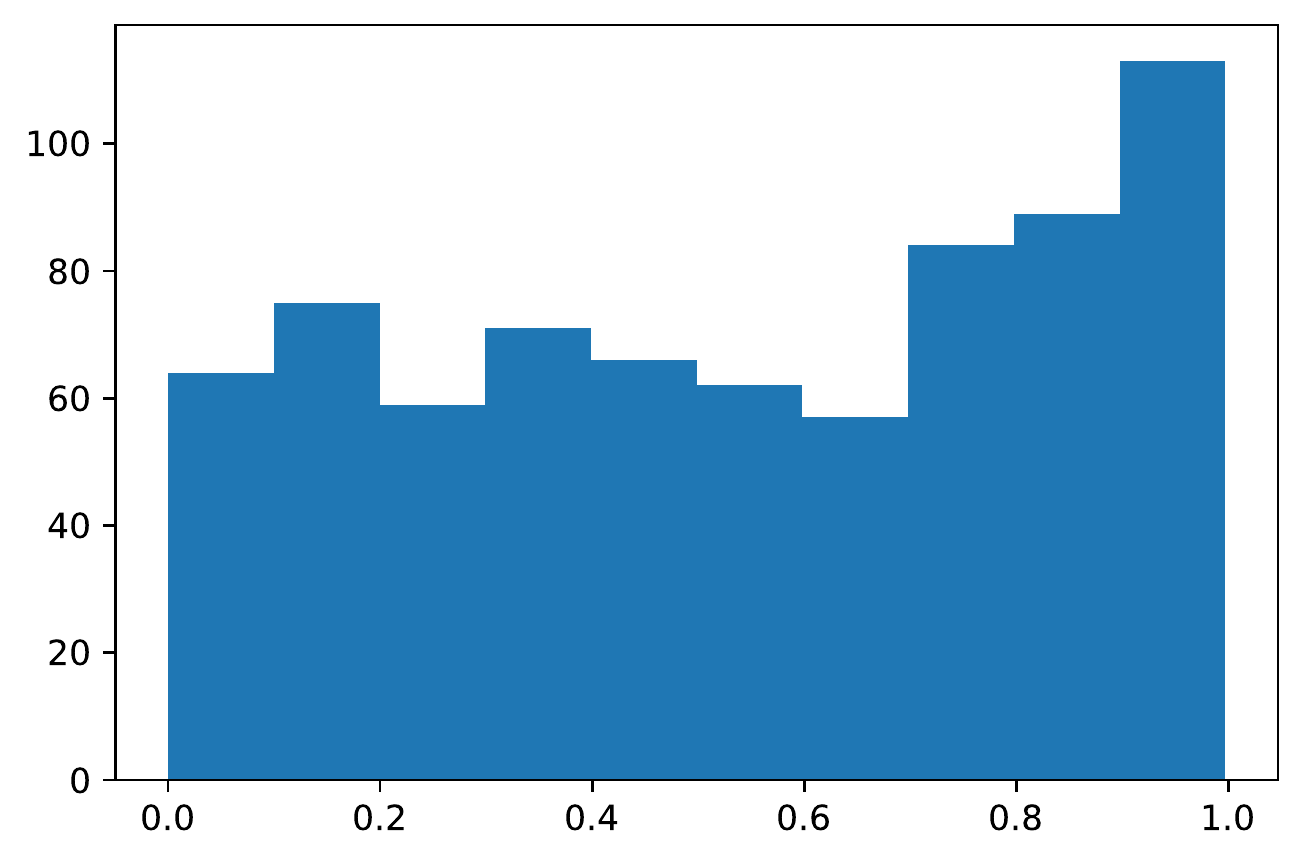}
    \includegraphics[width=0.49\linewidth]{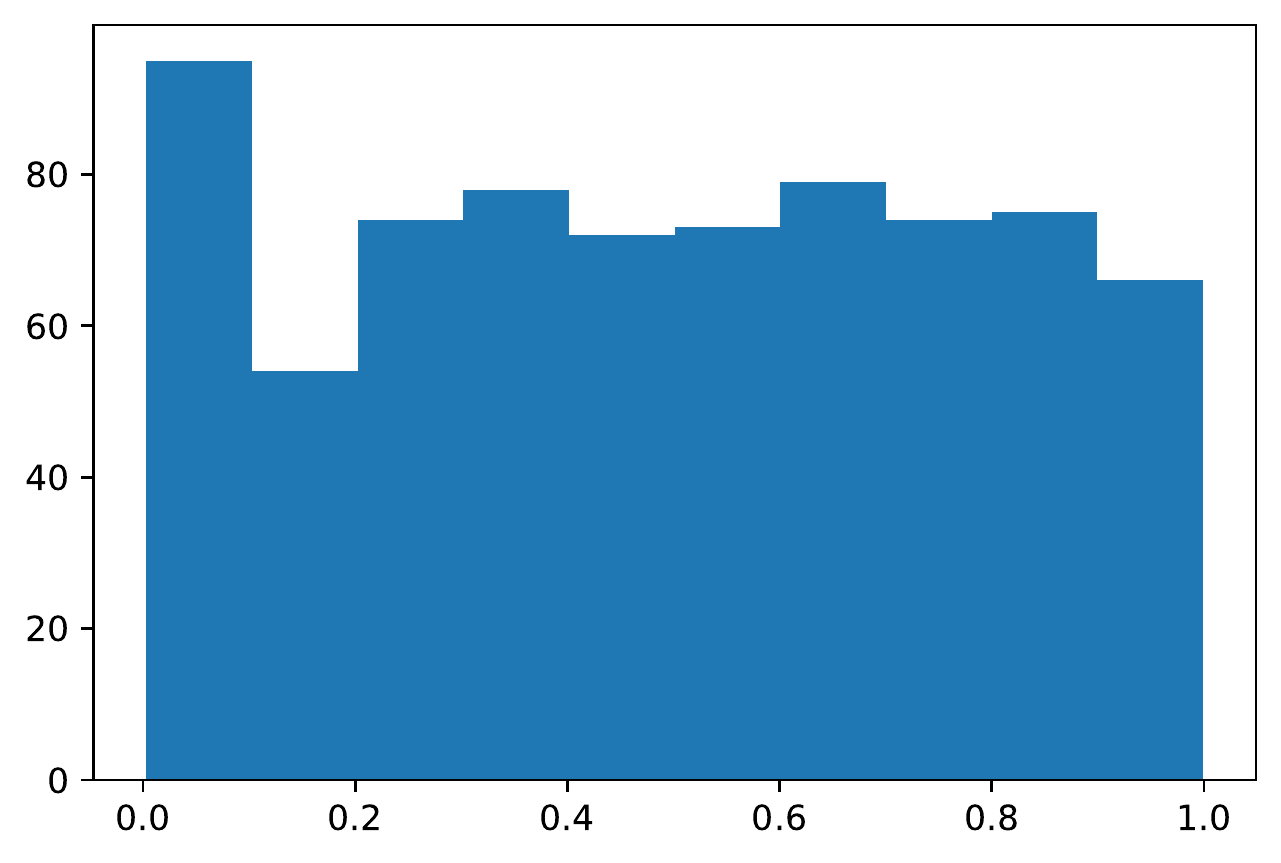}
  \end{center}
  \caption{Left panel:
    The histogram (with 10 bins) of the p-values at the last step
    for the nonnegative conformal martingale of Fig.~\ref{fig:Absenteeism}
    on the \texttt{Absenteeism} dataset.
    Right panel:
    The p-values of the same nonnegative conformal martingale
    on the randomly permuted \texttt{Absenteeism} dataset.}
  \label{fig:histograms}
\end{figure}

Let us apply these ideas to the \texttt{Absenteeism} dataset
briefly described in Section~\ref{sec:introduction}.
We will estimate the distribution of the past p-values using a histogram
and then will use the estimated distribution for betting
(therefore, we will implicitly assume the stability of the distribution of p-values).
The dataset consists of 740 observations (employees' absences) and a number of attributes;
it is given in the Clustering section of the repository,
and so there is no specified label.
Let us use four attributes,
Age, Disciplinary failure, Education, and Son (meaning the number of children);
the other attributes appeared to me more subjective
(such as Social drinker and Social smoker) or less relevant.
To apply the same nonconformity measure as before,
\eqref{eq:1-NN},
we need to nominate one of the attributes as label,
and ``Disciplinary failure'' appears to be particularly relevant
for studying the phenomenon of absenteeism.
To make the attributes comparable, let us divide Age by 50,
Education by 3 (this attribute ranges from 1, high school, to 4, master and doctor),
and Son by 4.
With this choice, the histogram of p-values at the last step is given
in the left panel of Fig.~\ref{fig:histograms},
and it is visibly non-uniform, namely tends to increase from let to right
(unlike the more stable right panel showing the p-values for a randomly permuted dataset;
cf.\ Remark~\ref{rem:validity}).

\begin{figure}[t]
  \begin{center}
    \includegraphics[width=0.49\linewidth]{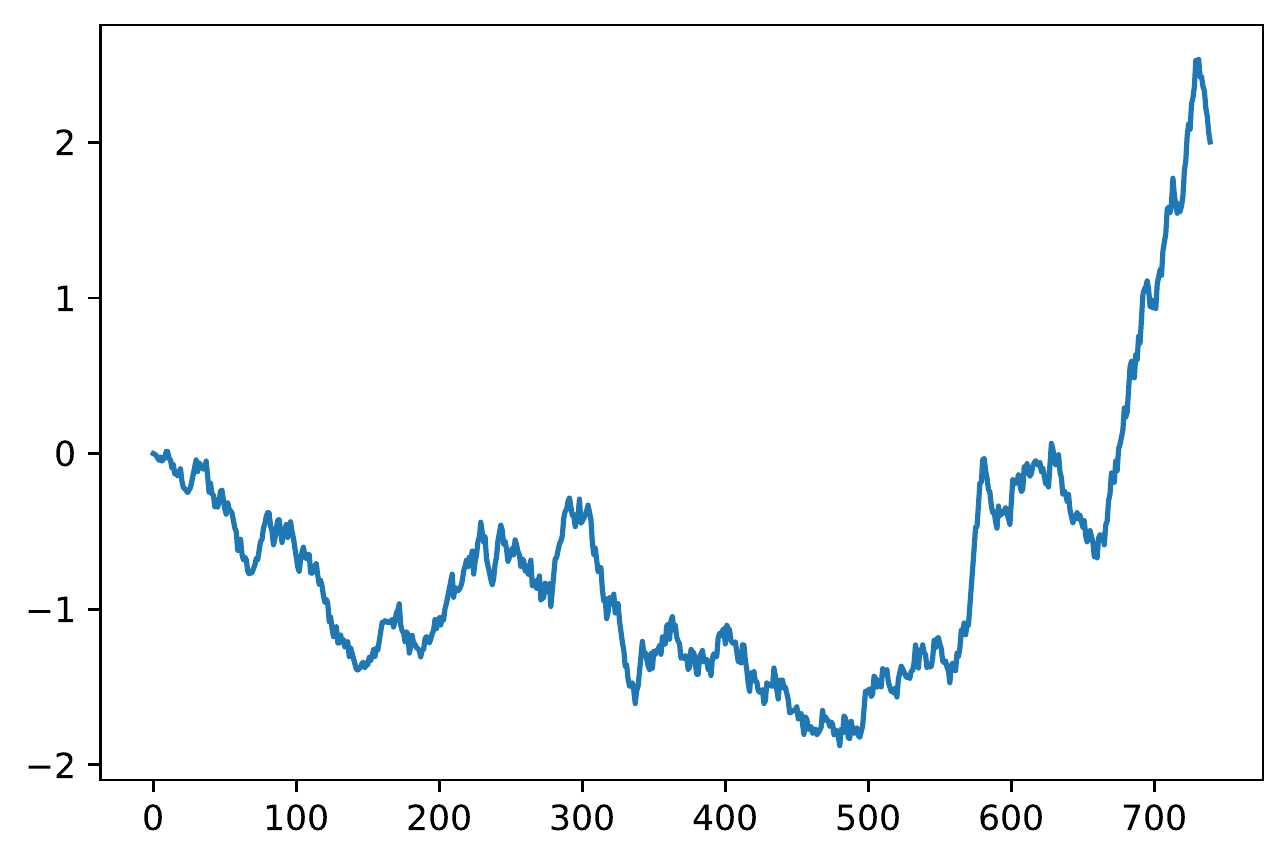}
    \includegraphics[width=0.49\linewidth]{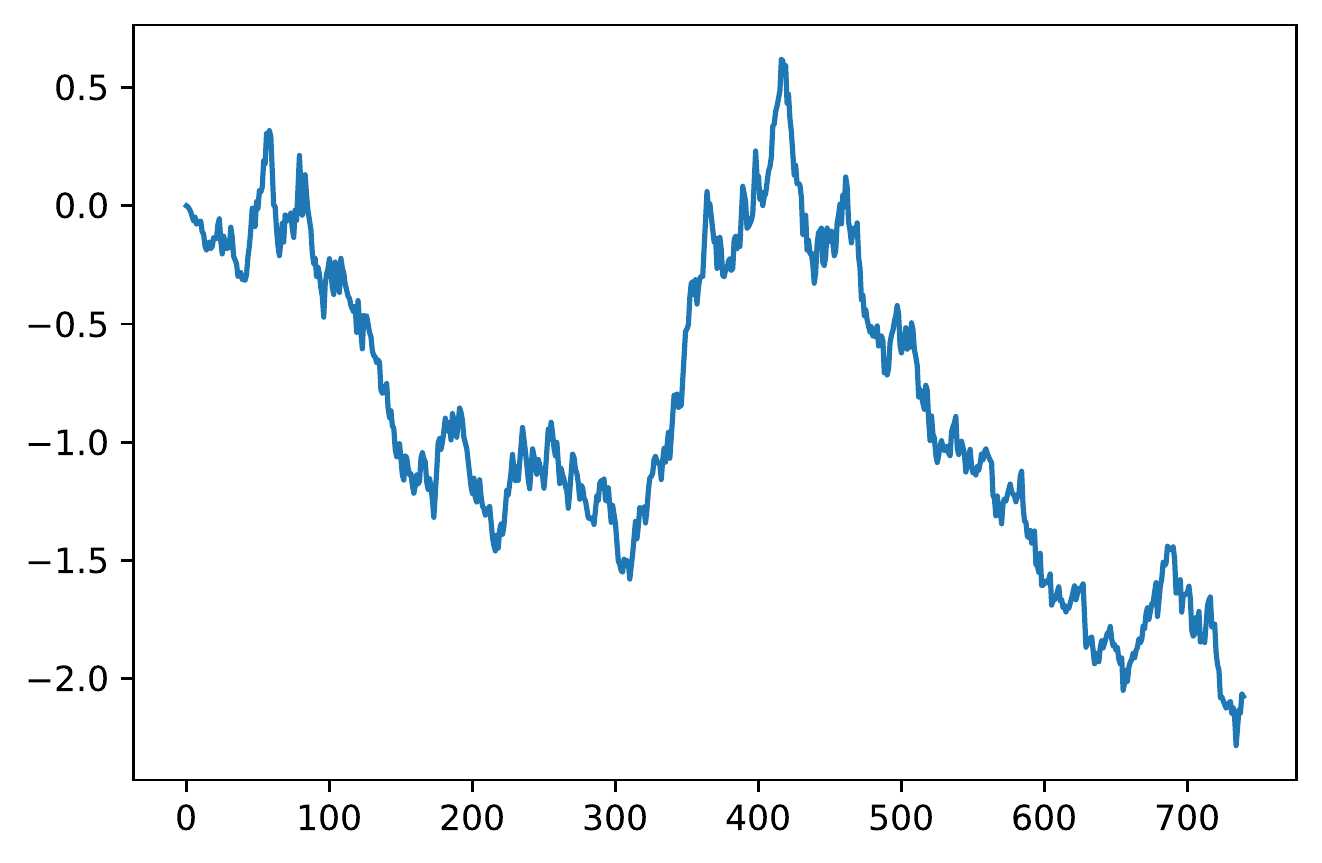}
  \end{center}
  \caption{Left panel:
    The values $S_n$ of a nonnegative conformal martingale
    after observing the first $n$ absences, $n=0,\dots,740$, in the \texttt{Absenteeism} dataset,
    with the log-10 scale for the vertical axis.
    The initial value $S_0$ is 1, and the final value $S_{740}$ is $100.50$.
    Right panel:
    The values $S_n$ of the same nonnegative conformal martingale as in the left panel
    after observing the first $n$ absences in a randomly permuted \texttt{Absenteeism} dataset,
    with the log-10 scale for the vertical axis.
    The initial value $S_0$ is 1, and the final value $S_{740}$ is
    $0.00841$.}
  \label{fig:Absenteeism}
\end{figure}

Figure~\ref{fig:Absenteeism} shows the results for a simple conformal martingale
exploiting the non-uniformity of the p-values.
We maintain $B$ bins corresponding to the subintervals $I_i:=[(i-1)/B,i/B]$ of $[0,1]$,
$i=1,\dots,B$ (we ignore the possibility of p-values landing on a boundary between two bins).
Initially, each bin contains $C$ p-values
(these dummy p-values are an element of regularization).
At step $n$, after computing the $n$ p-values for this step
(using the nonconformity measure \eqref{eq:1-NN}),
the algorithm puts them into the corresponding bins
and uses as its betting function $f_n$ the function equal to $(C+n_i)/(C+n/B)$ on the $i$th bin,
where $n_i$ is the number of p-values in that bin.
The mean of the betting function is 1 (and $C+n/B$ is the normalizing constant),
and it approximates Kelly gambling.
The conformal martingale whose performance is shown in Fig.~\ref{fig:Absenteeism}
is \eqref{eq:bet};
its initial value is 1.

The expectation of the martingale's final value is 1,
and we can use Jeffreys's \cite[Appendix~B]{Jeffreys:1961} rule of thumb for interpreting
the amount of evidence against the null hypothesis of randomness that it provides:
\begin{itemize}
\item
  A value below 1 supports the null hypothesis.
\item
  A value in the interval $(1,\sqrt{10})$ provides poor evidence against the null hypothesis
  (is not worth more than a bare mention).
\item
  A value in $(\sqrt{10},10)$ provides substantial evidence.
\item
  For a value in $(10,10^{3/2})$, the evidence is strong.
\item
  For a value in $(10^{3/2},100)$, the evidence is very strong.
\item
  Finally, for values above $100$ the evidence is decisive.
\end{itemize}
For example, the evidence against the hypothesis of randomness
provided by the left-hand panel of Fig.~\ref{fig:Absenteeism},
which uses $B=C=10$, is definitive
(although bordering on very strong)
since the final value of the conformal martingale is $100.50$.
Varying various parameters leads to similar, often stronger, results.
For example, replacing the ratio in~\eqref{eq:1-NN} by difference,
adding two extra attributes Social drinker and Social smoker,
and setting $B=C=20$,
leads to a final value of $3446.75$
greatly exceeding the threshold for definitive evidence.
However, all of these conformal martingales find deviations from randomness
only at the end of the period (starting from the 650th observation
at the earliest).

\subsection*{Optimality of sequential testing procedures}

The efficiency of our procedure for testing randomness
will be the topic of Section~\ref{sec:probability},
and in this section we will only discuss the nature of the problem.
Perhaps the most satisfactory results about the efficiency of sequential testing procedures
are optimality results such as that obtained by Wald and Wolfowitz \cite{Wald/Wolfowitz:1948}
for Wald's \cite{Wald:1945,Wald:1947} sequential probability ratio test.
The goal of establishing such optimality results for our procedures for testing randomness
would be, however, too ambitious.

Wald's sequential probability ratio test was designed by him in April 1943
\cite[Section B]{Wald:1945}
for the problem of testing a simple hypothesis against a simple alternative,
with IID observations.
Let $S_n$ be the likelihood ratio of the alternative hypothesis to the null hypothesis
after $n$ observations.
The test consists in fixing two positive constants $A$ and $B$ such that $A>B$
and stopping as soon as $S_n$ leaves the interval $(B,A)$.
If $S_n>A$ at that time $n$, we reject the null hypothesis;
otherwise, we accept it.

A sequential test can make errors of two kinds:
reject the null hypothesis when it is true
(error of the first kind)
or accept it when it is false
(error of the second kind).
Any sequential probability ratio test $T$ is efficient in a strong sense:
if another sequential test $T'$ has errors of the first and second kind
that are not worse than those for $T$,
the expected time of reaching a decision is as good for $T$ as it is for $T'$
(or better),
under both null and alternative hypotheses.
In other words, sequential probability ratio tests optimize
the number of observations needed to arrive at a decision,
under natural constraints.

Wald showed the efficiency of his test in the sections
``Efficiency of the Sequential Probability Ratio Test''
in \cite{Wald:1945,Wald:1947}
ignoring the possibility of $S_n$ overshooting $A$ or undershooting $B$.
In \cite{Wald/Wolfowitz:1948} he and Wolfowitz provided a full proof.

The strength of this result is made possible by the restricted nature of the testing problem.
Both null and alternative hypotheses are known probability distributions.
The test is specified by two numbers, $A$ and $B$.
The situation with testing randomness using conformal martingales
is very different.
A conformal martingale is determined by the underlying nonconformity measure,
which can even involve an element of intelligence.
See, e.g., \cite{Vovk/etal:2005book}, which defines numerous nonconformity measures
based on powerful algorithms of machine learning,
including neural networks.
We cannot expect to be able to \emph{prove} that such a procedure is successful
(as argued by philosophers in other contexts; see, e.g., \cite[Section 20]{Popper:1982}).
The task of designing a conformal martingale is too open-ended for that.

Our approach to establishing the efficiency of our strongly valid procedures
will not be based on optimality.
The idea is to show that our procedures do not constrain us:
whatever a procedure for testing randomness can achieve,
can be achieved with conformal martingales.
Notice that even the Wald--Wolfowitz result can be interpreted in this way.
However, our results in Section~\ref{sec:probability} will be much cruder.

\section{Multistage nonrandomness detection}
\label{sec:SR}

Our main concern in this section is application of conformal prediction to change detection,
which we already started discussing in Section~\ref{sec:introduction}.
A typical example of change detection is where we observe attacks,
which we assume to be IID, on a computer system.
When a new kind of attacks appears, the process of attacks ceases to be IID,
and we would like to raise an alarm soon afterwards.
The two benchmark datasets that we considered in the previous section
can also be used to illustrate the problem of change detection:
we may be interested in deciding when to retrain a predictor
and in detecting a change in the pattern of workforce absences.

There is vast literature on change detection;
see, e.g., \cite{Poor/Hadjiliadis:2009,Shiryaev:2019-disorder} for reviews.
However, the standard case is where the pre-change and post-change distributions are known,
and the only unknown is the time of change.
Generalizations of this picture usually stay fairly close to it
(see, e.g., \cite[Section 7.3]{Poor/Hadjiliadis:2009}).
Conformal change detection relaxes the standard assumptions radically.

As explained in Section~\ref{sec:introduction},
we may regard any problem of detecting nonrandomness in the online mode
as a problem of detecting a change point.
The latter includes as special case the situation where the assumption of randomness is never satisfied,
since 0 is an allowed change point.
Our informal goal is to raise an alarm as soon as possible after the hypothesis of randomness
ceases to be true.
In the previous section we did not insist on having an explicit rule for raising an alarm,
and simply regarded the value of a nonnegative conformal martingale starting from 1
as the amount of evidence found against the hypothesis of randomness,
but in this section it will be more convenient to couch our discussion
in terms of such rules.

As already mentioned,
the kind of guarantees provided by the policy of raising an alarm when $S_n/S_0\ge c$
is often regarded as too strong to be really useful.
This can be illustrated using the analogue of the left panel of Figure~\ref{fig:USPS}
for a randomly permuted USPS dataset.
The same conformal martingale performs as shown in the right panel of Figure~\ref{fig:USPS}
(this is Figure 7.8 in \cite{Vovk/etal:2005book}).
The conformal martingale is trying to gamble against an exchangeable sequence of observations,
which is futile,
and so its value decreases exponentially quickly.
If a change occurs at some point in the distant future,
it might take a long time for the martingale to recover its value.
This is a general phenomenon;
we must pay for giving ourselves the chance to detect lack of exchangeability
by losing capital in the situation of exchangeability.

Weaker guarantees are provided by multistage procedures originated, in a basic form,
by Shewhart in his control chart techniques \cite{Shewhart:1931}
and perfected by Page \cite{Page:1954} and Kolmogorov and Shiryaev \cite{Kolmogorov/Shiryaev:1960}
(see also the fascinating historical account in \cite[Section~1]{Shiryaev:2010}).

\subsection*{CUSUM-type change detection}

A standard multistage procedure of raising alarms is the CUSUM procedure
proposed by Page \cite{Page:1954}
(see also \cite[Section 6.2]{Poor/Hadjiliadis:2009}).
According to this procedure,
we raise the $k$th alarm at the time
\begin{equation}\label{eq:CUSUM}
  \tau_k
  :=
  \min
  \left\{
    n>\tau_{k-1}
    \st
    \max_{i=\tau_{k-1},\dots,n-1}
    \frac{S_n}{S_i}
    \ge
    c
  \right\},
  \quad
  k\in\N,
\end{equation}
where the threshold $c>1$ is a parameter of the algorithm,
$\tau_0:=0$, and $\min\emptyset:=\infty$.
If $\tau_k=\infty$ for some $k$, an alarm is raised only finitely often;
otherwise it is raised infinitely often.
The procedure is usually applied to the likelihood ratio process
between two power distributions,
but it can be applied to any positive martingale,
and in this paper we are interested in the case where $S$ is a conformal martingale,
which is now additionally assumed to be positive,
ensuring that the denominator in \eqref{eq:CUSUM} is always non-zero.
CUSUM is often interpreted as a repeated sequential probability ratio test
\cite[Section 4.2]{Page:1954}.
The \emph{conformal CUSUM procedure}
(i.e., CUSUM applied to a positive conformal martingale)
was introduced in \cite{Volkhonskiy/etal:2017COPA};
however, a basic and approximate version of this procedure
has been known since 1990: see \cite{McDonald:1990}.

Properties of validity for the conformal CUSUM procedure
will be obtained in this paper
as corollaries of the corresponding properties of validity
for the Shiryaev--Roberts procedure, which we consider next.

\subsection*{Shiryaev--Roberts change detection}

A popular alternative to the CUSUM procedure
is the Shiryaev--Roberts procedure \cite{Shiryaev:1963,Roberts:1966},
which modifies \eqref{eq:CUSUM} as follows:
\begin{equation}\label{eq:SR}
  \tau_k
  :=
  \min
  \left\{
    n>\tau_{k-1}
    \st
    \sum_{i=\tau_{k-1}}^{n-1}
    \frac{S_n}{S_i}
    \ge
    c
  \right\},
  \quad
  k\in\N
\end{equation}
(i.e., we just replace the $\max$ in \eqref{eq:CUSUM} by $\sum$).
We will again apply it to a conformal martingale $S$,
still assumed to be positive,
obtaining the \emph{conformal Shiryaev--Roberts procedure}.

The procedure defining $\tau_1$ is based on the statistics
\begin{equation}\label{eq:R}
  R_n
  :=
  \sum_{i=0}^{n-1}
  \frac{S_n}{S_i},
\end{equation}
which admit the recursive representation
\begin{equation}\label{eq:RR}
  R_n
  =
  \frac{S_n}{S_{n-1}}
  \left(
    R_{n-1} + 1
  \right),
  \quad
  n\in\N,
\end{equation}
with $R_0:=0$.
An interesting finance-theoretic interpretation of this representation
is that $R_n$ is the value at time $n$ of a portfolio that starts from $\$0$ at time 0
and invests $\$1$ into the martingale $S$ at each time $i=1,2,\dots$
\cite[Section 2]{Du/etal:2017}.
If and when an alarm is raised at time $n$,
we apply the same procedure to the remaining observations $z_{n+1},z_{n+2},\dots$.

The following proposition gives a non-asymptotic property of validity of the Shiryaev--Roberts procedure.
Roughly, it says that we do not expect the first alarm to be raised too soon
under the hypothesis of randomness.

\begin{proposition}\label{prop:SR-E}
  The conformal Shiryaev--Roberts procedure \eqref{eq:SR}
  satisfies $\Expect(\tau_1)\ge c$, for any $c>1$,
  under the assumptions of Proposition~\ref{prop:validity}.
\end{proposition}

Of course, we can apply Proposition~\ref{prop:SR-E} to other alarm times as well
obtaining $\Expect(\tau_k-\tau_{k-1})\ge c$ for all $k\in\N$
(and similar inequalities for some conditional expectations,
as discussed below in the proof of Proposition~\ref{prop:SR-lim}).
Therefore, more generally,
the time interval between raising successive alarms is not too short
under the hypothesis of randomness.

All results of this section
(from Proposition~\ref{prop:SR-E} to Corollary~\ref{cor:CUSUM-lim})
are general and applicable to any positive martingale $S$.
However, they are usually stated for $S$ being the likelihood ratio
between two power distributions (pre-change and post-change).
To simplify exposition,
I will state them only for $S$ being a positive conformal martingale
with the underlying filtration $(\GGG_n)_{n=0,1,\dots}$,
where $\GGG_n$ is generated by the first $n$ p-values $p_1,\dots,p_n$.
However, our arguments (which are standard in literature on change detection)
will be applicable to any filtration and any positive martingale with respect to that filtration.

\begin{proof}[Proof of Proposition~\ref{prop:SR-E}]
  The proof will follow from the fact that $R_n-n$ is a martingale;
  this fact (noticed, in a slightly different context, in \cite[Theorem~1]{Pollak:1987})
  follows from \eqref{eq:RR}:
  since $S$ is a martingale,
  \[
    \Expect(R_n \mid \GGG_{n-1})
    =
    \frac{\Expect(S_n \mid \GGG_{n-1})}{S_{n-1}}
    \left(
      R_{n-1} + 1
    \right)
    =
    R_{n-1} + 1.
  \]
  Another condition for $R_n-n$ being a martingale requires the integrability of $R_n$,
  which follows from the integrability of each addend in \eqref{eq:R}:
  \[
    \Expect
    \left(
      \frac{S_n}{S_i}
    \right)
    =
    \Expect
    \left(
      \Expect
      \left(
        \frac{S_n}{S_i}
	\mid
	\GGG_i
      \right)
    \right)
    =
    \Expect(1)
    =
    1
    <
    \infty.
  \]

  Fix the threshold $c>1$.
  By Doob's optional sampling theorem (see, e.g., \cite[Chapter 7, Section 2, Theorem 1]{Shiryaev:2019})
  applied to the martingale $R_n-n$,
  \[
    \Expect(\tau_1)
    =
    \Expect(R_{\tau_1})
    \ge
    c.
  \]
  Applying this theorem, however, requires some regularity conditions,
  and the rest of this proof is devoted to checking technical details.

  If $\tau_1=\infty$ with a positive probability,
  we have $\Expect(\tau_1)=\infty\ge c$,
  and so we assume that $\tau_1<\infty$ a.s.
  Doob's optional sampling theorem is definitely applicable to the stopping time $\tau_1\wedge L$,
  where $L$ is a positive constant
  (see, e.g., \cite[Chapter 7, Section 2, Corollary 1]{Shiryaev:2019}),
  and so the nonnegativity of $R$ implies
  \[
    \Expect(\tau_1)
    \ge
    \Expect(\tau_1\wedge L)
    =
    \Expect(R_{\tau_1\wedge L})
    \ge
    \Expect(R_{\tau_1} 1_{\{\tau_1\le L\}})
    \ge
    c \Prob(\tau_1\le L)
    \to
    c
  \]
  as $L\to\infty$, $1_E$ being the indicator function of an event $E$.
\end{proof}

\begin{corollary}\label{cor:CUSUM-E}
  The conformal CUSUM procedure \eqref{eq:CUSUM}
  also satisfies $\Expect(\tau_1)\ge c$
  under the assumptions of Proposition~\ref{prop:validity}.
\end{corollary}

\begin{proof}
  All our properties of validity for the CUSUM procedure
  will be deduced from the corresponding properties for Shiryaev--Roberts
  and the fact that Shiryaev--Roberts raises alarms more often than CUSUM does,
  in the following sense.
  Let $\tau_k$ (resp.\ $\tau'_k$) be the time of the $k$th alarm
  raised by Shiryaev--Roberts (resp.\ CUSUM).
  Then $\tau_k\le\tau'_k$ for all $k$;
  this can be checked by induction in~$k$.
\end{proof}

The next proposition is an asymptotic counterpart of Proposition \ref{prop:SR-E}
given in terms of frequencies.

\begin{proposition}\label{prop:SR-lim}
  Let $A_n$ be the number of alarms
  \[
    A_n
    :=
    \max\{k\st\tau_k\le n\}
  \]
  raised by the conformal Shiryaev--Roberts procedure \eqref{eq:SR}
  after seeing the first $n$ observations $z_1,\dots,z_n$.
  Then, under the assumptions of Proposition~\ref{prop:validity},
  \begin{equation}\label{eq:A-LLN}
    \limsup_{n\to\infty}
    \frac{A_n}{n}
    \le
    \frac1c
    \quad
    \text{a.s.}
  \end{equation}
\end{proposition}

Under the assumptions of Proposition~\ref{prop:validity},
all alarms are false,
and so \eqref{eq:A-LLN} limits the frequency of false alarms.

\begin{proof}
  Fix a positive conformal martingale $S$ and a threshold $c>0$.
  We can rewrite~\eqref{eq:SR} as
  \begin{equation}\label{eq:SR-R}
    \tau_k
    :=
    \min
    \left\{
      n>\tau_{k-1}
      \st
      R^k_n
      \ge
      c
    \right\},
  \end{equation}
  where
  \begin{equation*}
    R^k_n
    :=
    \sum_{i=\tau_{k-1}}^{n-1}
    \frac{S_n}{S_i}.
  \end{equation*}
  It will be convenient to modify \eqref{eq:SR-R} by forcing an alarm $L$ steps
  after the last one:
  \begin{equation*}
    \tau'_k
    :=
    (\tau'_{k-1}+L)
    \wedge
    \min
    \left\{
      n>\tau'_{k-1}
      \st
      R^{\prime k}_n
      \ge
      c
    \right\},
  \end{equation*}
  where $\tau'_0:=0$ and
  \begin{equation*}
    R^{\prime k}_n
    :=
    \sum_{i=\tau'_{k-1}}^{n-1}
    \frac{S_n}{S_i}.
  \end{equation*}
  (The value of $L$ will be chosen later.)
  Similarly to the proof of Corollary~\ref{cor:CUSUM-E},
  by induction in $k$ we can check that, for all $k$, $\tau'_k\le\tau_k$.

  We still have a recursive representation similar to \eqref{eq:RR} for ($R^k$ and) $R^{\prime k}$.
  Notice that $R^{\prime k}_n$, $n\ge\tau'_{k-1}$, is a nonnegative submartingale
  with $n-\tau'_{k-1}$ as its compensator
  (and we can set $R^{\prime k}_n$ and its compensator to 0 for $n<\tau'_{k-1}$).

  Remember that $\GGG_n$ is the $\sigma$-algebra generated by the p-values $p_1,\dots,p_n$,
  and let $\GGG_{\tau'_k}$ be the $\sigma$-algebra of events $E$ such that
  $E\cap\{\tau'_k\le n\}\in\GGG_n$ for all $n$
  (informally, $\GGG_{\tau'_k}$ consists of the events $E$
  expressible in terms of the p-values and settled at time $\tau'_k$).

  Let us say that $k\in\N$ is \emph{slow} if
  \[
    \Prob
    \left(
      \tau'_k - \tau'_{k-1} = L
      \mid
      \GGG_{\tau'_{k-1}}
    \right)
    \ge
    c/L;
  \]
  otherwise, $k$ is \emph{fast}.
  Notice that the event that $k$ is fast (or slow) is $\GGG_{\tau'_{k-1}}$-measurable.
  By Doob's optional sampling theorem and the nonnegativity of $R^{\prime k}_n$,
  where $n\ge\tau'_{k-1}$,
  for a fast $k$ we obtain, similarly to the proof of Proposition~\ref{prop:SR-E},
  \begin{align*}
    \Expect&
    \left(
      \tau'_k - \tau'_{k-1}
      \mid
      \GGG_{\tau'_{k-1}}
    \right)
    =
    \Expect
    \left(
      R^{\prime k}_{\tau'_k}
      \mid
      \GGG_{\tau'_{k-1}}
    \right)\\
    &=
    \Expect
    \left(
      R^{\prime k}_{\tau'_k}
      1_{\{\tau'_k - \tau'_{k-1} = L\}}
      \mid
      \GGG_{\tau'_{k-1}}
    \right)
    +
    \Expect
    \left(
      R^{\prime k}_{\tau'_k}
      1_{\{\tau'_k - \tau'_{k-1} < L\}}
      \mid
      \GGG_{\tau'_{k-1}}
    \right)\\
    &\ge
    0
    +
    c
    \Expect
    \left(
      1_{\{\tau'_k - \tau'_{k-1} < L\}}
      \mid
      \GGG_{\tau'_{k-1}}
    \right)
    \ge
    c(1-c/L)
    =
    c-c^2/L.
  \end{align*}

  Let $F\subseteq\N$ be the random set of all fast $k$,
  $S:=\N\setminus F$ be the random set of all slow $k$,
  and $F_K$ (resp.\ $S_K$) be the set consisting of the $K$ smallest elements
  of $F$ (resp.\ $S$).
  The strong law of large numbers for bounded martingale differences now implies
  \begin{equation}\label{eq:S}
    \liminf_{K\to\infty}
    \frac1K
    \sum_{k\in S_K}
    (\tau'_k-\tau'_{k-1})
    \ge
    L(c/L)
    =
    c
    \quad
    \text{a.s.}
  \end{equation}
  and
  \begin{equation}\label{eq:F}
    \liminf_{K\to\infty}
    \frac1K
    \sum_{k\in F_K}
    (\tau'_k-\tau'_{k-1})
    \ge
    c - c^2/L
    \quad
    \text{a.s.};
  \end{equation}
  the inequality in \eqref{eq:S} (resp.\ \eqref{eq:F})
  is interpreted as true when $\left|S\right|<\infty$
  (resp.\ $\left|F\right|<\infty$).
  Combining \eqref{eq:S} and \eqref{eq:F},
  we obtain
  \[
    \liminf_{K\to\infty}
    \frac{\tau'_K}{K}
    =
    \liminf_{K\to\infty}
    \frac1K
    \sum_{k=1}^K
    (\tau'_k-\tau'_{k-1})
    \ge
    c - c^2/L
    \quad
    \text{a.s.}
  \]
  Therefore, setting
  \[
    A'_n
    :=
    \max\{k\st\tau'_k\le n\},
  \]
  we have
  \[
    \limsup_{n\to\infty}
    \frac{A_n}{n}
    \le
    \limsup_{n\to\infty}
    \frac{A'_n}{n}
    \le
    \frac{1}{c - c^2/L},
  \]
  and it remains to let $L\to\infty$.
\end{proof}

\begin{remark}
  It might be tempting to deduce \eqref{eq:A-LLN} from Proposition~\ref{prop:SR-E}
  directly using a suitable law of large numbers.
  However, a simple application of the Borel--Cantelli--L\'evy lemma
  shows that we cannot do so without using the specifics of our stopping times $\tau_k$.
  Indeed, assuming $c\in\{2,3,\dots\}$,
  we can define a filtered probability space and stopping times $\tau_k$, $k=0,1,\dots$,
  with $\tau_0:=0$, in such a way that
  \[
    \tau_k - \tau_{k-1}
    =
    \begin{cases}
      1 & \text{with probability $1-k^{-2}$}\\
      (c-1)k^2+1 & \text{with probability $k^{-2}$}
    \end{cases}
  \]
  for all $k\in\N$
  (where the probabilities may be conditional on a suitable $\sigma$-algebra $\GGG_{\tau_{k-1}}$).
  Then $\Expect(\tau_k-\tau_{k-1})=c$ (and $\Expect(\tau_k-\tau_{k-1}\mid\GGG_{\tau_{k-1}})=c$)
  for all $k$ but, almost surely, $\tau_k-\tau_{k-1}=1$ from some $k$ on.
\end{remark}

Of course, the statement of Proposition~\ref{prop:SR-lim} also holds for the CUSUM procedure.

\begin{corollary}\label{cor:CUSUM-lim}
  Let $A_n$ be the number of alarms raised by the conformal CUSUM procedure \eqref{eq:CUSUM}
  after seeing the observations $z_1,\dots,z_n$.
  Then \eqref{eq:A-LLN} holds under the assumptions of Proposition~\ref{prop:validity}.
\end{corollary}

\begin{proof}
  As in the proof of Corollary~\ref{cor:CUSUM-E},
  combine Proposition~\ref{prop:SR-lim}
  with the fact that Shiryaev--Roberts raises alarms more often than CUSUM does.
\end{proof}

\subsection*{Optimality of procedures for change detection}

It is remarkable that both CUSUM and Shiryaev--Roberts procedures are optimal
under some natural conditions and for some natural criteria of optimality.
As already mentioned,
in standard settings of change detection
the task is to detect a change from one known power probability distribution for the incoming data
to another known power probability distribution.
CUSUM and Shiryaev--Roberts procedures are then applied to a specific martingale,
the likelihood ratio of the post-change distribution to the pre-change distribution.
Therefore, they depend on just one parameter, $c$
(for given pre-change and post-change distributions),
whereas in the context of testing randomness we have a wide class
of CUSUM and Shiryaev--Roberts procedures,
built on top of different conformal martingales.

The five standard criteria for the quality of such specific procedures have been referred to
by the letters A--E;
see, e.g., Shiryaev \cite{Shiryaev:2019-disorder}.
Under natural conditions,
Shiryaev--Roberts is optimal under two of the criteria,
and CUSUM is optimal under one of them.
Such statements of optimality are very satisfactory results
about the efficiency of the corresponding procedures.

In this paper we only discuss validity results for CUSUM and Shiryaev--Roberts
in the context of randomness testing and do not claim their optimality.
As discussed at the end of Section~\ref{sec:martingales},
this is a difficult task already for the basic strongly valid testing procedure
using conformal martingales.
The null hypothesis (that of randomness) is composite and, moreover, very large
(for large $\mathbf{Z}$), and we do not specify any alternatives;
we simply do not have enough structure to specify a meaningful optimization problem.

\section{IID probability vs exchangeability probability}
\label{sec:IID-exchangeability}

We will be discussing two related interpretations of the efficiency of conformal martingales:
on one hand, they detect deviations from randomness,
and on the other hand, they detect deviations from exchangeability.
This section makes a break in our discussion of martingales,
and here we will explore the relation between randomness and exchangeability
for finite binary sequences.
(Remember that for infinite sequences they are connected by de Finetti's theorem,
as discussed in Remark~\ref{rem:de-Finetti};
for finite sequences their relation becomes much less close.)

In the 1960s Kolmogorov started revival of the interest in random sequences,
believing that they are important for understanding the applications of probability theory and statistics.
As already mentioned,
he concentrated on binary sequences (as a simple starting point),
in which context he often referred to them as \emph{Bernoulli sequences}.
His first imperfect publication on this topic was the 1963 paper \cite{Kolmogorov:1963}
(Kolmogorov refers to it as ``incomplete discussion'',
according to the English translation of \cite{Kolmogorov:1965-Latin}).
In the same year he conceived using the notion of computability for formalizing randomness.
Kolmogorov's main publications on the algorithmic theory of randomness
were \cite{Kolmogorov:1965-Latin,Kolmogorov:1968-Latin,Kolmogorov:1983-Latin}.

Let $\Omega:=\{0,1\}^N$ be the set of all binary sequences of a given length $N$,
interpreted as sequences of observations.
The time horizon $N\in\N$ can be regarded as fixed in the rest of this paper
(apart from \eqref{eq:log-scale} and the appendix).

Let $B_p$ be the Bernoulli probability measure on $\{0,1\}$ with the probability of 1 equal to $p\in[0,1]$:
$B_p(\{1\}):=p$.
The \emph{upper IID probability} of a set $E\subseteq\Omega$ is defined to be
\begin{equation}\label{eq:IID-prob}
  \UiidP(E)
  :=
  \sup_{p\in[0,1]}
  B_p^N(E),
\end{equation}
and the \emph{upper exchangeability probability} of $E\subseteq\Omega$ is defined to be
\begin{equation}\label{eq:exch-prob}
  \UEP(E)
  :=
  \sup_P
  P(E),
\end{equation}
$P$ ranging over the exchangeable probability measures on $\Omega$
(in the current binary case we can say
that a probability measure $P$ on $\Omega$ is \emph{exchangeable}
if $P(\{\omega\})$ depends on $\omega\in\Omega$ only via the number of 1s in~$\omega$).

\begin{remark}
  The \emph{lower probabilities} corresponding to~\eqref{eq:IID-prob} and~\eqref{eq:exch-prob}
  are $1-\UiidP(\Omega\setminus E)$ and $1-\UEP(\Omega\setminus E)$, respectively.
  In this paper we never need lower probabilities.
\end{remark}

The function $\UiidP$ can be used when testing the hypothesis of randomness:
if $\UiidP(E)$ is small (say, below 1\%)
and the observed sequence $\omega$ is in $E$ that is chosen in advance,
we can reject the hypothesis that the observations in $\omega$ are IID.
Similarly, $\UEP$ can be used when testing the hypothesis of exchangeability.
This is an instance of application of \emph{Cournot's principle},
often regarded to be the only bridge between probability theory and its applications.
The principle was widely discussed at the beginning of the 20th century
and defended by, e.g., Borel, L\'evy, and Kolmogorov \cite[Section 2.2]{Shafer/Vovk:2006SS}.
Kolmogorov's statement of Cournot's principle in his \emph{Grundbegriffe}
\cite[Chapter I, Section 2]{Kolmogorov:1933} is
\begin{quote}
  If $\textsf{P}(A)$ is very small,
  then one can be practically certain that the event $A$ will not occur
  on a single realization of the conditions $\mathfrak{S}$.
\end{quote}
(The conditions $\mathfrak{S}$ in this quote refer to the probability trial under discussion.)
In the form stated by Kolmogorov, the principle goes back to Jacob Bernoulli \cite{Bernoulli:1713}
(see, e.g., \cite[Section 2.2]{Shafer/Vovk:2006SS}).
It establishes a bridge between probability theory and our expectations about reality;
observing an event $A$ (assumed to be chosen in advance) of a small probability casts doubt on $\textsf{P}$.
Cournot's \cite[p.~78]{Cournot:1843} contribution was to state that this is the only bridge
between probability theory and reality.

Cournot's principle suggests the following understanding of the efficiency
of a method of testing the hypothesis of randomness:
given any event $E$ such that $\UiidP(E)$ is very small,
the method should allow us to reject the hypothesis of randomness
after observing $E$.

From the point of view Bayesian statistics,
a possible justification of Cournot's principle
is the possibility of transforming p-values to Bayes factors;
e.g., $p\mapsto2\sqrt{p}$ is a \emph{calibrator}, in the sense of transforming p-values into Bayes factors
in favor of the null hypothesis (small values serving as evidence against the null hypothesis).
For details, see, e.g., \cite[Section 6]{Shafer/etal:2011}.
Therefore, small p-values transform into small Bayes factors.

\begin{remark}
  More generally,
  the function $p\mapsto p^{1-\kappa}/\kappa$ (the reciprocal to \eqref{eq:betting-function})
  is a calibrator,
  for any $\kappa>0$ \cite[Section 6]{Shafer/etal:2011}.
  A popular way of getting rid of the parameter $\kappa$,
  popularized by Berger and his co-authors
  (see, e.g., \cite{Benjamin/Berger:2019}) is to minimize over it,
  getting
  \[
    \min_{\kappa\in(0,1)}
    p^{1-\kappa}/\kappa
    =
    -e p \log p.
  \]
  The reciprocal to this value (corresponding to Bayes factors against the null hypothesis)
  is sometimes (e.g., in the JASP package) referred to as the Vovk--Sellke maximum p-ratio
  (it was introduced independently in \cite{Vovk:1993logic} and \cite{Sellke/etal:2001}).
  It is not useful for us in this paper since it is not a valid Bayes factor.
\end{remark}

\begin{proposition}\label{prop:1}
  For any $E\subseteq\Omega$,
  \begin{equation}\label{eq:1}
    \UiidP(E) \le \UEP(E) \le 1.5 \sqrt{N} \UiidP(E).
  \end{equation}
\end{proposition}

\begin{proof}
  The first inequality in \eqref{eq:1}
  follows from each power probability measure on $\Omega$ being exchangeable.
  If $E$ contains either the all-0 sequence $0\dots0$ or the all-1 sequence $1\dots1$,
  the second inequality in \eqref{eq:1} is obvious
  ($\UiidP(E)=\UEP(E)=1$).
  If $E$ is empty, it is also obvious
  ($\UiidP(E)=\UEP(E)=0$).
  Finally, if $E$ is nonempty and contains neither sequence,
  we have, for some $k\in\{1,\dots,N-1\}$,
  \begin{align}
    \UEP(E)
    &=
    \UEP(E\cap\Omega_k)
    =
    \frac{1/\binom{N}{k}}{(k/N)^k(1-k/N)^{N-k}}
    \UiidP(E\cap\Omega_k)\label{eq:1-1}\\
    &\le
    \frac{k!(N-k)!N^N}{N!k^k(N-k)^{N-k}}
    \UiidP(E)
    \le
    \sqrt{2\pi}e^{1/6}
    \sqrt{\frac{k(N-k)}{N}}
    \UiidP(E)\label{eq:1-2}\\
    &\le
    (\sqrt{2\pi}e^{1/6}/2)
    \sqrt{N}
    \UiidP(E)
    \le
    1.5
    \sqrt{N}
    \UiidP(E),\label{eq:1-3}
  \end{align}
  where $\Omega_k$ is the set of all sequences in $\Omega$ containing $k$ 1s.
  The first equality in \eqref{eq:1-1} follows from each exchangeable probability measure on $\Omega$
  being a convex mixture of the uniform probability measures on $\Omega_k$, $k=0,\dots,N$.
  The second equality in \eqref{eq:1-1} follows from the maximum of $B_p(\{\omega\})$,
  $\omega\in\Omega_k$, over $p\in[0,1]$
  being attained at $p=k/N$.
  The first inequality in \eqref{eq:1-2} is equivalent to the obvious
  $\UiidP(E\cap\Omega_k)\le\UiidP(E)$.
  The second inequality in \eqref{eq:1-2} follows from Stirling's formula
  \begin{equation}\label{eq:Stirling}
    n!
    =
    \sqrt{2\pi} n^{n+1/2} e^{-n} e^{r_n},
    \quad
    0 < r_n < \frac{1}{12n},
  \end{equation}
  valid for all $n\in\N$;
  see, e.g., \cite{Robbins:1955},
  where it is also shown that $r_n > \frac{1}{12n+1}$.
  The first inequality in \eqref{eq:1-3} follows from $\max_{p\in[0,1]}p(1-p)=1/4$.
\end{proof}

\begin{remark}
  The constant 1.5 in inequality \eqref{eq:1} is not too far from being optimal:
  when $N=2$ and $E=\{(0,1)\}$,
  it can be improved only to $\sqrt{2}\approx1.414$.
  Notice that our argument in fact gives $\sqrt{2\pi}e^{1/6}/2\approx1.481$
  instead of $1.5$.
\end{remark}

Kolmogorov's \cite{Kolmogorov:1968-Latin,Kolmogorov:1983-Latin} implicit interpretation of~\eqref{eq:1}
was that $\UiidP$ and $\UEP$ are close;
on the log scale we have
\begin{equation}\label{eq:log-scale}
  -\log\UiidP(E) = -\log\UEP(E) + O(\log N),
\end{equation}
whereas typical values of $-\log\UiidP(E)$ and $-\log\UEP(E)$ have the order of magnitude $N$
for small (but non-zero) $\left|E\right|$.
See Appendix~\ref{sec:ATR} for further details.

From the point of view of Cournot's principle,
Proposition~\ref{prop:1} may be interpreted as saying that there is not much difference
between testing randomness and testing exchangeability.
If we have a test with critical region $E$ of size $\epsilon$ for testing exchangeability,
we can use it for testing randomness and its size will not increase;
in the opposite direction,
if we have a test with critical region $E$ of size $\epsilon$ for testing randomness,
we can use it for testing exchangeability, and its size will increase to at most $1.5 \sqrt{N} \epsilon$.
On the log-scale of Equation \eqref{eq:log-scale}
the difference between the evidence provided by $E$
against the hypothesis of randomness and against the hypothesis of exchangeability
is $O(\log N)$;
it is clear that the left-hand side of \eqref{eq:log-scale} can be as large as $N$
(for a non-empty $E$ and assuming that the logarithms are binary).
In the algorithmic theory of randomness it is customary to ignore such differences,
although from the point of view of statistics, the difference is substantial.

\section{Conformal probability}
\label{sec:probability}

In this section we explore the efficiency of conformal martingales,
restricting ourselves to the simple case of a finite horizon $N$
(as in the previous section).
First we will define upper conformal probability $\UCP$,
an analogue of $\UiidP$ and $\UEP$ for testing randomness using conformal martingales.
Our simple version of upper conformal probability will be sufficient for our current purpose;
there are other natural definitions.
The \emph{upper conformal probability} of $E\subseteq\Omega$ is
\begin{equation}\label{eq:UCP}
  \UCP(E)
  :=
  \inf\{S_0\st\forall(z_1,\dots,z_N)\in E:S_N(z_1,\theta_1,z_2,\theta_2,\dots)\ge1 \text{ $\theta$-a.s.}\},
\end{equation}
where $S$ ranges over the nonnegative conformal martingales,
``$\theta$-a.s.''\ refers to the uniform probability measure over $(\theta_1,\theta_2,\dots)\in[0,1]^{\infty}$,
and $S_0$ stands for the constant $S_0(z_1,\theta_1,z_2,\theta_2,\dots)$.
The definition \eqref{eq:UCP} is in the spirit of \cite[Section~2.1]{Shafer/Vovk:2019};
$\UCP(E)<\epsilon$ for a small $\epsilon>0$ means that there exists a nonnegative conformal martingale
with a small initial value, below $\epsilon$,
that almost surely increases its value manyfold, to at least 1,
if the event $E$ happens.
Therefore, we do not expect this event to happen under the hypothesis of randomness.
This is spelled out in the following lemma.

\begin{lemma}
  For any event $E$,
  $\UiidP(E)\le\UCP(E)$.
\end{lemma}

\begin{proof}
  Let $P$ be a power probability measure on $\Omega$
  and $S$ be a nonnegative conformal martingale
  satisfying the condition in \eqref{eq:UCP}.
  It suffices to prove
  \begin{equation}\label{eq:goal}
    P(E)\le S_0;
  \end{equation}
  indeed, we can then obtain $\UiidP(E)\le\UCP(E)$
  by taking $\sup$ of the left-hand side of \eqref{eq:goal} over $P$
  and taking $\inf$ of the right-hand side over $S$.

  To check \eqref{eq:goal},
  remember that $S_N\ge1_E$ a.s., where $1_E$ is the indicator of $E$.
  Since $S$ is a nonnegative martingale under $P$
  (like any nonnegative conformal martingale),
  we have
  \[
    P(E) = \int 1_E \dd P \le \int S_N \dd P = S_0.
    \qedhere
  \]
\end{proof}

We will use upper conformal probability to make the notion of efficiency
for conformal martingales more precise.
Namely, if $\UiidP$ and $\UCP$ are shown to be close,
this could be interpreted as conformal martingales being able to detect any deviations from randomness.
By Cournot's principle,
any deviations from randomness are demonstrated by indicating in advance an event $E$
of small probability under any power probability measure,
i.e., such that $\UiidP(E)$ is small,
which then happens.
If $\UiidP$ and $\UCP$ are close, $\UCP(E)$ will also be small.
This means that there exists a nonnegative conformal martingale $S$
that increases its value manyfold when $E$ happens.
We can choose $S$ in advance since $E$ is chosen in advance.
And such an $S$ will be successful whenever $E$ is.

The following proposition shows that $\UEP$ and $\UCP$ are close,
in the sense similar to the closeness of $\UiidP$ and $\UEP$ asserted in Proposition~\ref{prop:1}
(see also~\eqref{eq:log-scale}).

\begin{proposition}\label{prop:2}
  For any $E\subseteq\Omega$,
  \begin{equation}\label{eq:2}
    \UEP(E) \le \UCP(E) \le N\UEP(E).
  \end{equation}
\end{proposition}

Proposition~\ref{prop:2} is a statement of efficiency for conformal martingales.
It says that, at our crude scale,
lack of exchangeability can be detected using conformal martingales.
Namely, given a critical region $E$ of a very small size $\epsilon:=\UEP(E)$,
we can construct a nonnegative conformal martingale
with initial capital $N\epsilon$ or less that attains capital of at least $1$ when $E$ happens.

Combining the right-hand sides of \eqref{eq:1} and \eqref{eq:2} we obtain
\begin{equation*}
  \UCP(E) \le 1.5 N^{1.5} \UiidP(E).
\end{equation*}
This inequality says that conformal martingales are efficient at detecting deviations
not only from exchangeability but also from randomness.
Given a critical region $E$ of a very small size $\epsilon:=\UiidP(E)$,
there exists a nonnegative conformal martingale
that increases an initial capital of $1.5 N^{1.5}\epsilon$
to at least $1$ when $E$ happens.

\begin{proof}[Proof of Proposition~\ref{prop:2}]
  First we check the left inequality in \eqref{eq:2}.
  We will do even more:
  we will check that it remains true
  even if the right-hand side of \eqref{eq:UCP}
  is replaced by
  \begin{equation}\label{eq:stronger}
    \inf\{S_0\st\forall(z_1,\dots,z_N)\in E:\Expect_{\theta} S_N(z_1,\theta_1,z_2,\theta_2,\dots)\ge1\},
  \end{equation}
  where the $\Expect_{\theta}$ refers to the uniform probability measure
  over $(\theta_1,\theta_2,\dots)\in[0,1]^{\infty}$.
  Notice that $S_0,\dots,S_N$ in \eqref{eq:UCP} is a martingale
  in the filtration $(\GGG_n)$ generated by the p-values $p_1,\dots,p_N$
  under any exchangeable probability measure on $\Omega$;
  this follows from the fact that $p_1,\dots,p_N$ are IID and uniform on $[0,1]$
  under any exchangeable probability measure
  (see Remark~\ref{rem:validity}).
  Therefore, for each $E\subseteq\Omega$
  and each nonnegative conformal martingale $S$ such that $\Expect_{\theta} S_N\ge1_E$,
  we have
  \begin{equation}\label{eq:we-have}
    \Prob_z(E)
    \le
    \Prob_z(\Expect_{\theta}S_N\ge1)
    \le
    \Expect_z(\Expect_{\theta} S_N)
    =
    \Expect_{z,\theta} S_N
    =
    S_0,
  \end{equation}
  where $\Prob_z$ refers to $(z_1,\dots,z_N)\sim P$,
  $P$ is an exchangeable probability measure on $\Omega$,
  $\Expect_{\theta}$ refers to $(\theta_1,\dots,\theta_N)\sim U^N$,
  $U$ is the uniform probability measure on $[0,1]$,
  and $\Expect_{z,\theta}$ refers to $(z_1,\dots,z_N)\sim P$ and $(\theta_1,\dots,\theta_N)\sim U^N$ independently.
  Taking the $\sup$ of the leftmost expression in \eqref{eq:we-have} over $P$
  and the $\inf$ of the rightmost expression in \eqref{eq:we-have} over $S$,
  we obtain the left inequality in \eqref{eq:2}.

  It remains to check the right inequality in \eqref{eq:2}.
  Let us first check the part ``$\le$'' of the first equality in
  \begin{equation*}
    \UCP(\{\omega\})
    =
    \frac{k!(N-k)!}{N!}
    =
    \UEP(\{\omega\}),
  \end{equation*}
  where $k\in\{0,\dots,N\}$ and $\omega\in\Omega$ contains $k$ 1s
  (the part ``$\ge$'' was established in the previous paragraph;
  it will not be used in the rest of this proof).

  Let $\omega=(z_1,\dots,z_N)$ be the representation of $\omega$ as a sequence of bits.
  Consider the nonnegative conformal martingale $S^{\omega}$
  obtained from the identity nonconformity measure $A(z_1,\dots,z_n):=(z_1,\dots,z_n)$
  and a betting martingale $F$ such that $F(\Box)=1/\binom{N}{k}$
  (where $\Box$ is the empty sequence) and
  \[
    \frac{F(p_1,\dots,p_{n-1},p_n)}{F(p_1,\dots,p_{n-1})}
    :=
    \begin{cases}
      \frac{n}{k_n} & \text{if $p_n\le k_n/n$ and $z_n=1$}\\
      \frac{n}{n-k_n} & \text{if $p_n\ge k_n/n$ and $z_n=0$}\\
      0 & \text{otherwise},
    \end{cases}
  \]
  where $n=1,\dots,N$ and $k_n$ is the number of 1s in $\omega$ observed so far,
  \[
    k_n
    :=
    \left|
      \{j\in\{1,\dots,n\}\st z_j=1\}
    \right|;
  \]
  in particular, $k_N=k$.
  (Intuitively,
  $S^{\omega}$ gambles recklessly on the $n$th observation being $z_n$.)
  If the actual sequence of observations happens to be $\omega$,
  on step $n$ the value of the martingale $S^{\omega}$ is multiplied, a.s.,
  by the fraction whose numerator is $n$
  and whose denominator is the number of bits $z_n$ observed in $\omega$ so far.
  The product of all these fractions over $n=1,\dots,N$
  will have $N!$ as its numerator and $k!(N-k)!$ as its denominator.
  This conformal martingale is almost deterministic,
  in the sense of not depending on $\theta_n$ provided $\theta_n\notin\{0,1\}$,
  and its final value on $\omega$ is, a.s.,
  \[
    \frac{1}{\binom{N}{k}}
    \frac{N!}{k!(N-k)!}
    =
    1.
  \]

  To move from singletons to arbitrary $E\subseteq\Omega$,
  notice that a finite linear combination of conformal martingales $S^{\omega}$
  with positive coefficients
  is again a conformal martingale,
  since they involve the same nonconformity measure, and betting martingales can be combined.
  Fix $E\subseteq\Omega$ and remember that $\Omega_k$ is the set of all sequences in $\Omega$
  containing $k$ 1s.
  Represent $E$ as the disjoint union
  \[
    E = \bigcup_{k=0}^N E_k,
    \quad
    E_k \subseteq \Omega_k,
  \]
  and let $U_k$ be the uniform probability measure on $\Omega_k$.
  We then have
  \begin{multline*}
    \UCP(E)
    \le
    \sum_{\omega\in E}
    \UCP(\{\omega\})
    =
    \sum_{k=0}^N
    \sum_{\omega\in E_k}
    \UCP(\{\omega\})
    =
    \sum_{k=0}^N
    \sum_{\omega\in E_k}
    \UEP(\{\omega\})\\
    =
    \sum_{k=0}^N
    U_k(E_k)
    \le
    N
    \max_{k=0,\dots,N}
    U_k(E_k)
    =
    N\UEP(E),
  \end{multline*}
  where the last inequality holds when, e.g., $E$ does not contain the all-0 sequence $0\dots0\in\Omega$.
  If $E$ does contain the all-0 sequence,
  it is still true that
  \[
    \UCP(E)
    \le
    1
    \le
    N
    =
    N\UEP(E).
    \qedhere
  \]
\end{proof}

\section{Conclusion}
\label{sec:conclusion}

This paper gives a review of known methods of testing randomness online,
all of which are based on conformal martingales.
It raises plenty of questions, without giving many answers.
Propositions~\ref{prop:1} and~\ref{prop:2}
say that IID, exchangeability, and conformal upper probabilities are close,
but the accuracy of these statements is very low
and far from meaningful in practice.
The most obvious direction of further research is to obtain more accurate results
(a simple example related to Proposition~\ref{prop:1} will be given in Appendix~\ref{sec:ATR}).
It would be ideal to establish exact bounds on upper conformal probability
in terms of upper IID probability and upper exchangeability probability.
The most natural definition of upper conformal probability in this context
might involve randomness in a more substantial way than our official definition \eqref{eq:UCP} does
(cf., e.g., \eqref{eq:stronger}).

\subsection*{Acknowledgements}

The original version of this paper was written in support of my poster at ISIPTA 2019
(Eleventh International Symposium on Imprecise Probabilities: Theories and Applications)
presented on 5 July 2019.
I am grateful to all discussants,
including Thomas Dietterich, Wouter Koolen, and Glenn Shafer.
This research has been supported by Astra Zeneca and Stena Line.
The experiments with the tangent metric are based on Daniel Keysers's \cite{Keysers:2000} implementation in C.
Bartels's test for randomness used in Section~\ref{sec:batch}
is from the R package \texttt{randtests} (Testing randomness in R)
by Frederico Caeiro and Ayana Mateus~\cite{Mateus/Caeiro:2014}.
Comments by three reviewers have been extremely useful for improving the presentation;
they led, in particular, to the inclusion of experiments with the \texttt{Absenteeism} dataset
and to the explicit discussion of Cournot's principle.

\appendix
\section{Connections with the algorithmic theory of randomness}
\label{sec:ATR}

The emphasis of this appendix will be on Kolmogorov's approach to randomness and exchangeability
expressed in Martin-L\"of's \cite{Martin-Lof:1966} terms,
which are closer to the traditional statistical language.
(Kolmogorov's original definitions, equivalent but given in terms of algorithmic complexity,
will be discussed in Appendix~\ref{sec:ATC}.)
In our terminology we will follow \cite{Vovk/Vyugin:1993}.
Following Kolmogorov \cite{Kolmogorov:1965-Latin,Kolmogorov:1968-Latin,Kolmogorov:1983-Latin},
we will only consider the case of binary observations.

A \emph{measure of randomness} is an upper semicomputable function $f:\{0,1\}^*\to[0,1]$
such that, for any $N\in\N$, any power probability measure $P$ on $\{0,1\}^N$,
and any $\epsilon>0$, we have \eqref{eq:p-variable}.
The upper semicomputability of $f$ means that there exists an algorithm
that, when fed with a rational number $r$ and sequence $\omega\in\{0,1\}^*$,
eventually stops if $f(\omega)<r$ and never stops otherwise.

In other words, a measure of randomness
is a family of p-variables for testing randomness in $\{0,1\}^N$.
The requirement of upper semicomputability is natural:
e.g., if $f(\omega)<0.01$ (the p-value is highly statistically significant),
we should learn this eventually.

Analogously, a \emph{measure of exchangeability} is an upper semicomputable function $f:\{0,1\}^*\to[0,1]$
such that, for any $N\in\N$, any exchangeable measure $P$ on $\{0,1\}^N$,
and any $\epsilon>0$, we have \eqref{eq:p-variable}.

\begin{lemma}\label{lem:universal}
  There exists a measure of randomness $f$ (called \emph{universal})
  such that any other measure of randomness $f'$ satisfies $f=O(f')$.
  There exists a measure of exchangeability $f$ (called \emph{universal})
  such that any other measure of exchangeability $f'$ satisfies $f=O(f')$.
\end{lemma}

The proof of Lemma~\ref{lem:universal} is standard;
see, e.g., \cite{Martin-Lof:1966} or \cite[Lemma~4]{Vovk/Vyugin:1993}.

In the algorithmic theory of randomness,
it is customary to measure lack of randomness or exchangeability on the log scale.
Therefore, we fix a universal measure of randomness $f$,
set $\diid:=-\log f$,
and refer to $\diid(\omega)$ as the \emph{deficiency of randomness} of the sequence $\omega\in\{0,1\}^*$.
Similarly, we fix a universal measure of exchangeability $f$,
set $\dE:=-\log f$,
and refer to $\dE(\omega)$ as the \emph{deficiency of exchangeability} of $\omega$.
(Traditionally, the log is binary.)

Proposition~\ref{prop:1} immediately implies
\begin{equation}\label{eq:close}
  \dE(\omega) - O(1) \le \diid(\omega) \le \dE(\omega) + \frac12 \log N + O(1),
\end{equation}
where $\omega$ ranges over $\{0,1\}^*$ and $N$ is the length of $\omega$.
In fact, we can interpret \eqref{eq:close} as the algorithmic version of Proposition~\ref{prop:1}.
Kolmogorov regarded the coincidence to within $\log$ as close enough,
at least for some purposes:
cf.\ the last two paragraphs of \cite{Kolmogorov:1968-Latin};
therefore, he preferred the simpler definition
$\dE(\omega)\approx0$ of $\omega$ being a Bernoulli sequence.

Proposition~\ref{prop:1} is very crude,
and Section~\ref{sec:conclusion} sets the task of obtaining more accurate results.
In fact, such results are known in the context of the algorithmic theory of randomness;
they were obtained in the paper \cite{\OCMXV} written under Kolmogorov's supervision.

To clarify relations between algorithmic randomness and exchangeability,
we will need another notion, binomiality.
The binomial probability distribution $\bin_{N,p}$ on $\{0,\dots,N\}$ with parameter $p$
is defined by
\[
  \bin_{N,p}(\{k\})
  :=
  \binom{N}{k}
  p^k (1-p)^{N-k},
  \quad
  k\in\{0,\dots,N\}.
\]
A \emph{measure of binomiality} is an upper semicomputable function
$f:\{(N,k)\st N\in\N,k\in\{0,\dots,N\}\}\to[0,1]$
such that, for any $N\in\N$, any $p\in[0,1]$,
and any $\epsilon>0$,
\begin{equation*}
  \bin_{N,p}(\{k\st f(N,k)\le\epsilon\})
  \le
  \epsilon.
\end{equation*}
\begin{lemma}
  There exists a measure of binomiality $f$ (called \emph{universal})
  such that any other measure of binomiality $f'$ satisfies $f=O(f')$.
\end{lemma}
We fix a universal measure of binomiality $f$,
set $\dbin(k;N):=-\log f(N,k)$,
and refer to $\dbin(k;N)$ as the \emph{deficiency of binomiality} of $k$ (in $\{0,\dots,N\}$).

\begin{proposition}\label{prop:Vovk}
  For any constant $\epsilon>0$,
  \begin{multline*}
    (1-\epsilon)
    \left(
      \dE(\omega) + \dbin(k;N)
    \right)
    -
    O(1)
    \le
    \diid(\omega)\\
    \le
    (1+\epsilon)
    \left(
      \dE(\omega) + \dbin(k;N)
    \right)
    +
    O(1),
  \end{multline*}
  $N$ ranging over $\N$,
  $\omega$ over $\{0,1\}^N$,
  and $k$ being the number of 1s in $\omega$.
\end{proposition}

Proposition~\ref{prop:Vovk} follows immediately from
(and is stated, in a more precise form, after) \cite[Theorem 1]{\OCMXV}.
It says, informally, that the randomness of $\omega$ is equivalent
to the conjunction of two conditions:
$\omega$ should be exchangeable, and the number of 1s in it should be binomial.
For example, suppose that $N$ is a large even number
and the number of 1s in $\omega\in\{0,1\}^N$ is $k=N/2$.
Then $\omega$ might be perfectly exchangeable whereas it will not be random
since it belongs to the set of all binary sequences with the number of 1s precisely $N/2$,
whose probability \eqref{eq:half} is small.
(Example~\ref{ex:counterexample} was based on this observation
expressed in a different language.)

\section{Connections with the algorithmic theory of complexity}
\label{sec:ATC}

In Appendix~\ref{sec:ATR} we gave the definition of randomness equivalent to Kolmogorov's
but used Martin-L\"of's language of statistical tests.
Kolmogorov himself used the language of algorithmic complexity,
nowadays known as Kolmogorov complexity.
Apart from his papers \cite{Kolmogorov:1965-Latin,Kolmogorov:1968-Latin,Kolmogorov:1983-Latin}
on this topic,
Kolmogorov was also a co-author of \cite{Kolmogorov/Uspensky:1987-Latin},
which was based on his ideas and publications,
although he did not see the final version of that paper
\cite[Introduction]{Kolmogorov/Uspensky:1987-Latin}
and did not take part in preparing the talk on which it was based
\cite[p.~380]{Uspensky:1993}.
Another valuable source is the record of his 1982 talk \cite{Kolmogorov:1983LNM}.
From now on I will assume knowledge of some basic notions of the theory of Kolmogorov complexity.

Kolmogorov's original notion of randomness for an element $\omega$ of a simple finite set $M$
was that $K(\omega)\approx-\log\left|M\right|$,
where $K$ is Kolmogorov complexity and $\log$ is binary log
(see \cite[Section 4]{Kolmogorov:1965-Latin}).
Martin-L\"of \cite{Martin-Lof:2005} modified this requirement to $K(\omega\mid M)\approx-\log\left|M\right|$,
where $K$ now stands for conditional Kolmogorov complexity.
In his 1968 paper \cite[Section 2]{Kolmogorov:1968-Latin}
Kolmogorov gave his alternative formalization of von Mises's random sequences,
with a reference to Martin-L\"of:
namely, Kolmogorov said that a binary sequence $\omega$ of length $N$ containing $k$ 1s is \emph{Bernoulli}~if
\[
  K(\omega\mid k,N)
  \approx
  \log
  \binom{N}{k}.
\]
It is natural to call the difference
\begin{equation}\label{eq:dE}
  \dE(\omega)
  :=
  \log
  \binom{N}{k}
  -
  K(\omega\mid k,N)
\end{equation}
the \emph{deficiency of exchangeability} of $\omega$
(in terminology close to that of \cite{Kolmogorov:1983LNM} and \cite{Kolmogorov/Uspensky:1987-Latin}).
This definition is equivalent to the one given in Appendix~\ref{sec:ATR}
(in the sense that the difference between \eqref{eq:dE}
and $\dE(\omega)$ as defined in Appendix~\ref{sec:ATR} is bounded in absolute value
by a constant independent of $\omega$;
this is the first italicized statement in \cite[p.~616]{Martin-Lof:1966}).

Being Bernoulli in the sense of Kolmogorov
does not fully reflect the intuition of being random,
i.e., being a plausible outcome of a sequence of $N$ tosses of a possibly biased coin.
This intuition is better captured by
\begin{equation}\label{eq:diid}
  \diid(\omega)
  :=
  \inf_{p\in[0,1]}
  \left(
    -\log B_p^N(\omega)
    -
    K(\omega\mid p,N)
  \right),
\end{equation}
which we call the \emph{deficiency of randomness} of $\omega$, being small.
This is equivalent to the definition given in Appendix~\ref{sec:ATR}.

Proposition~\ref{prop:Vovk} in Appendix~\ref{sec:ATR}
about the difference between \eqref{eq:dE} and \eqref{eq:diid}
can be made much more precise if we modify our definitions.
Theorem~1 of \cite{\OCMXV} shows that
\begin{equation*}
  \Diid(\omega)
  =
  \left(
    \log\binom{N}{k} - \KP\left(\omega\mid N,k,\Dbin(k;N)\right)
  \right)
  +
  \Dbin(k;N)
  +
  O(1),
\end{equation*}
where $N$ ranges over $\N$,
$\omega$ over $\{0,1\}^N$,
$k$ is the number of 1s in $\omega$,
$\KP$ is prefix complexity,
$\Diid$ is the analogue of $\diid$
using prefix instead of Kolmogorov complexity,
and $\Dbin(k;N)$ is the \emph{prefix deficiency of binomiality} of $k$ defined by
\[
  \Dbin(k;N)
  :=
  \inf_{p\in[0,1]}
  \left(
    -\log\bin_{N,p}(k)
    -
    \KP(k\mid p,N)
  \right),
\]
where $\bin_{N,p}$ is the binomial probability measure on $\{0,\dots,N\}$ with parameter $p$,
as defined earlier.
Theorem~2 of \cite{\OCMXV} characterizes $\Dbin(k;N)$ in terms of prefix complexity,
showing that it can be as large as $\frac12\log N+O(1)$.
\end{document}